\documentclass[12pt]{amsart}

\usepackage{amssymb,amsthm,amsmath,mathrsfs}
\usepackage{amsfonts}
\usepackage{graphicx,pstricks}
\usepackage{abstract}

\newtheorem{theorem}{Theorem}

\newtheorem{corollary}[theorem]{Corollary}

\newtheorem{lemma}[theorem]{Lemma}
\newtheorem{proposition}[theorem]{Proposition}
\theoremstyle{definition}
\newtheorem{example}[theorem]{Example}
\newtheorem{definition}[theorem]{Definition}
\newtheorem{remark}[theorem]{Remark}

\newcommand{\rank}{\operatorname{rk}}
\newcommand{\MSC}{\noindent \textit{2010 MSC:} }

\begin{document}

\title{Dissonant points and the region of influence of non-saddle sets  %Points dissonants et région d’influence des ensembles non-selle.
\footnote{The authors are supported by MINECO (MTM2012-30719). The first author is also supported by the FPI grant BES-2013-062675.}}

\author{H\'ector Barge}
\address{E.T.S. Ingenieros inform\'{a}ticos, Universidad Polit\'{e}cnica de Madrid,  Madrid 28660, Spain}
\email{h.barge@upm.es}

\author{Jos\'e M.R. Sanjurjo}
\address{Facultad de C. C. Matem\'aticas, Universidad Complutense de Madrid, Madrid 28040, Spain}
\email{email:jose\_sanjurjo@mat.ucm.es}

\maketitle

%\selectlanguage{english}

\begin{abstract}
 The aim of this paper is to study dynamical and topological properties of a
flow in the region of influence of an isolated non-saddle set. We see, in
particular, that some topological conditions are sufficient to guarantee
that these sets are attractors or repellers. We study in detail the
existence of dissonant points of the flow, which play a key role in the
description of the region of influence of a non-saddle set. These points are
responsible for much of the dynamical and topological complexity of the
system. We also study non-saddle sets from the point of view of the Conley
index theory and consider, among other things, the case of flows on
manifolds with trivial first cohomology group. For flows on these manifolds,
dynamical robustness is equivalent to topological robustness. We carry out a
particular study of 2-dimensional flows and give a topological condition
which detects the existence of dissonant points for flows on surfaces. We
also prove that isolated invariant continua of planar flows with global
region of influence are necessarily attractors or repellers.
\end{abstract}

%\selectlanguage{french}

%\begin{abstract}
%L'objectif de cet article est d'étudier les propriétés dynamiques et topologiques d'un flux  dans  la région d'influence d'un ensemble non-selle isolé.  Nous observons, en particulier, que certaines conditions topologiques  sont suffisantes pour garantir que ces ensembles soient des attracteurs  ou des répulseurs. Nous étudions en detail  l'existence des points dissonants du flux, qui  jouent un rôle clé  dans la description de la région d'influence d'un ensemble non-selle. Ces points sont responsables de la plupart de la complexité dynamique et topologique du système. Nous étudions  aussi  les ensembles non-selle du point de vue de la théorie de l'index de Conley et nous considérons, entre autres, le cas de varietés avec un premier groupe de cohomologie trivial. Pour les flux dans ces varietés la robustesse dynamique est équivalente à la robustesse topologique.  Nous réalisons une étude particulier des flux bidimensionnels et nous présentons une condition topologique qui détecte l'existence de points dissonants des flux sur des surfaces. Nous prouvons aussi que les continuums invariants isolés des flux dans le plan avec une région d'influence global son nécessairement des attracteurs ou des répulseurs.
%\end{abstract} 
%\vspace{3pt}

%\selectlanguage{english}

\MSC{37C75, 37B30.}
\vspace{3pt}


\section{Introduction}

This paper is devoted to the study of a flow near a compact invariant set. This was one of the first classical subjects dealt with by H. Poincar\'{e}, I.
Bendixson, A. Andronov and S. Lefschetz at the beginning of the qualitative theory of {di\-ffe\-ren\-tial} equations, with further contributions by authors such as D.M. Grobman, P. Hartman, J.K. Hale and A. Stokes and many others. T. Ura and I. Kimura were able to provide in \cite{Ura-Kimura} a rather general description. Wazewski's and Conley's index theories can be considered an important chapter of this general topic since they deal essentially with local properties of the flow near (isolated) {in\-va\-riant} sets. The local character of the flow is, to a large extent, determined by the perturbation properties of attraction and repulsion induced by the invariant set in its vicinity and, for this reason, we study in detail its region of influence. Isolated non-saddle sets are important in this context since they have an open region of influence. Moreover, this region has a particularly nice structure near the compactum since it is organized into a purely attracted part and a purely repelled part, where the properties of attraction and repulsion are uniform. In other words, isolated non-saddle sets have isolating blocks consisting of the asymptotic (positive and negative) parts only. Saddle and non-saddle sets were first
studied by N.P. Bhatia \cite{Bhatia} and by T. Ura \cite{Ura} but, according to Ura, they were introduced before by P. Seibert in an oral communication. The
theory of non-saddle sets can be considered as a general theory of stability and attraction, which extends the classical one and encompasses recent
developments such as the theory of unstable attractors with no external explosions \cite{Athanassopoulos, Moron, Sanchez-Gabites Transactions}.

We study in the paper the topological aspects of non-saddle sets and of their region of influence. We see, in particular, that some topological
conditions are sufficient to guarantee that these sets are attractors or repellers. We examine the general structure of a flow having an isolated non-saddle
set. This structure can be considerably more complicated than the one for attractors and repellers and even than the one for unstable attractors with no external
explosions. We show that the reason for this complexity lies in the existence of dissonant points in the region of influence of the non-saddle set. We
illustrate this phenomenon with an example based on a saddle-node bifurcation in the torus. We identify topological conditions which prevent
the existence of dissonant points. For instance, flows on connected closed manifolds $M$ with trivial first cohomology group $H^{1}(M)$ do not have dissonant
points and the region of influence of the non-saddle set has a nice description. We also give some conditions, in terms of the Conley index of the invariant set, for it to be non-saddle. This result is a source of many interesting examples of flows on manifolds admitting all degrees of dynamical complexity. We carry out a particular study of 2-dimensional flows. We describe the structure of a class of non-saddle sets in the torus. We find a condition, involving the Euler characteristic of the region of
influence, to detect the existence of dissonant points of flows on surfaces. We prove that isolated invariant continua of planar flows with global
region of influence are attractors or repellers (in particular, they enjoy either positive or negative stability). Some results about existence of fixed points for planar flows are also established. Finally, we study the continuation properties of non-saddle sets for parametrized families of flows. We prove that, for flows on compact connected and oriented differentiable manifolds with trivial 1-dimensional cohomology, dynamical robustness is equivalent to topological robustness and, as a consequence, the preservation of non-saddleness can be detected by topological means. It can be seen that non-saddle sets are involved in generalized Andronov-Poincar\'{e}-Hopf bifurcations (in the sense of \cite{Sanjurjo JDE}) but this will be the subject of a different paper. Our point of view is mainly topological and has deep connections with the Conley index theory. We stress
the richness of the topological aspects of dynamical
systems and differential equations, often presenting unexpected and strange properties. As was remarked
by Kennedy and Yorke in \cite{Kennedy-Yorke} \textquotedblleft bizarre topology is natural in
dynamical systems.\textquotedblright 

Through the paper we shall consider flows $\varphi:M\times\mathbb{R}\to M$ defined on locally compact metric spaces. However, we will often focus on the particular cases of flows defined on locally compact ANRs and $n$-dimensional manifolds. A metrizable space is said to be an \emph{absolute neighborhood retract} (or $M\in ANR$) if for every {ho\-meo\-mor\-phism} $h$ mapping $M$ onto a closed subset $h(M)$ of a metrizable space $X$ there is a neighborhood $U$ of $h(M)$ in $X$ such that $h(M)$ is a retract of $U$. The class of ANRs includes a lot of nice spaces such as polyhedra, CW-complexes, manifolds, etc... Besides, open subsets of ANRs and retracts of ANRs are also ANRs. For more information about the theory of retracts see \cite{Hu}.

  A form of homotopy theory, namely \emph{shape theory}, which is the most {sui\-ta\-ble} for the study of global topological properties in dynamics, will be {occa\-sio\-na\-lly used}. Although a deep knowledge of shape theory is not necessary to understand the paper we {re\-co\-mmend} to the reader Borsuk's monography \cite{Borsuk} for an exhaustive treatment of the subject and \cite{Kapitanski, Robin y Salamon,Robinson Shape, Sanjurjo Transactions,  Sanjurjo London, Sanjurjo JMA, Sanjurjo LN, Wang-Li} for a concise introduction and some applications to dynamical systems. 

We make use of some notions of algebraic topology. A good reference for this material is the book of Spanier \cite{Spanier}. We will use the notation $H_*$ and $H^*$ for the singular homology and cohomology respectively and we will denote by $\check{H}^*$ the \v Cech cohomology, all of them with integer coefficients unless otherwise specified. We recall that \v Cech and singular cohomology theories agree on ANRs and, from this fact, combined with a simple argument involving the long exact sequences of \v Cech and singular cohomology of a pair, the natural homomorphism between both cohomologies and the five lemma, it can be proved that it also holds for pairs of ANRs. Another nice property of \v Cech cohomology is that it is a shape invariant. A recent paper showing some  applications of homological techniques to dynamical systems is \cite{Li}.

If a pair of spaces $(X,A)$ satisfies that its cohomology $\check{H}^k(X,A)$ is finitely generated for each $k$ and is non-zero only for a finite number of values of $k$, its \emph{Euler characteristic} is defined as
\[
\chi(X,A)=\sum_{k\geq 0} (-1)^k\rank \check{H}^k(X,A).
\]

A useful property of the Euler characteristic which can be found in \cite{Spanier} is that 
 \[
\chi(X)=\chi(X,A)+\chi(A). 
 \]
 
 A different Euler characteristic can be defined by using singular homology and cohomology (both produce the same result by the Universal Coefficients Theorem) in the same way. This singular Euler characteristic agrees with the Euler characteristic previously defined for ANRs and pairs of ANRs by the previous remarks.

The main reference for the elementary concepts of dynamical systems will be \cite{Bhatia-Szego} but we also recommend \cite{Robinson dynamics, Palis, Pilyugin}. In the sequel we shall use the notation $\gamma(x)$ for the \emph{trajectory} of the point $x$, i.e. $\gamma(x)=\{xt\mid t\in\mathbb{R}\}$. Similarly for the \emph{positive semi-trajectory} $\gamma^+(x)=\{xt\mid t\in\mathbb{R}^+\}$ and the \emph{negative-semitrajectory} $\gamma^-(x)=\{xt\mid t\in\mathbb{R}^-\}$. By the \emph{omega-limit} of a point $x$ we understand the set $\omega(x)=\bigcap_{t>0}\overline{x[t,\infty)}$ while the \emph{negative omega-limit} is the set $\omega^*(x)=\bigcap_{t<0}\overline{x(-\infty,t]}$. On the other hand, the \emph{positive prolongational limit set} of a point $x$ is the set
\[
 J^+(x)=\bigcap_{U\in \mathcal{E}(x),t>0}\overline{U[t,\infty)},
\]
where $\mathcal{E}(x)$ denotes the system of neighborhoods of the point $x$. The \emph{negative prolongational limit set} of a point $x$, $J^-(x)$ is defined in an analogous fashion. Besides, we will need to introduce the concept of \emph{two-sided prolongational limit set}
\begin{definition}
Given $x\in M$, the \emph{two-sided prolongational limit set} of $x$, is defined to be
\begin{align*}
J^*(x):=\{(y,z)\in M\times M \mid \mbox{there exist}\; x_n\to x,\;t_n\to\infty,\;s_n\to-\infty\;\\
\mbox{such that},\; x_nt_n\to y\;\mbox{and}\;x_ns_n\to z\}.
\end{align*}
\end{definition}

It is easy to see that $J^*$ is closed in $M\times M$ and  that $(yt,zs)\in J^*(x)$ for every $(y,z)\in J^*(x)$ and every $t,s\in \mathbb{R}$.

Another notion which will be useful is that of \emph{parallelizable flow}. A flow $\varphi:M\times\mathbb{R}\to M$ is said to be \emph{parallelizable} if there exists a subset $S\subset M$ such that the map $h:S\times\mathbb{R}\to M$ defined by $(x,t)\mapsto xt$ is a homeomorphism. The set $S$ is called a \emph{section} of $\varphi$ and a direct consequence of the definition is that it is a strong deformation retract of $M$.

An important class of invariant compacta is the so-called \emph{isolated invariant sets} (see \cite{Conley,Conley-Easton,Easton} for details). These are compact invariant sets $K$ which possess an \emph{isolating neighborhood}, i.e. a compact neighborhood $N$  such that $K$ is the maximal invariant set in $N$. 

 A special kind of isolating neighborhoods will be useful in the sequel, the so-called \emph{isolating blocks}, which have good topological properties. More precisely, an isolating block $N$ is an isolating neighborhood such that there are compact sets $N^i,N^o\subset\partial N$, called the entrance and the exit sets, satisfying
\begin{enumerate}
\item $\partial N=N^i\cup N^o$;
\item for each $x\in N^i$ there exists $\varepsilon>0$ such that $x[-\varepsilon,0)\subset M-N$ and for each $x\in N^o$ there exists $\delta>0$ such that $x(0,\delta]\subset M-N$;
\item for each $x\in\partial N-N^i$ there exists $\varepsilon>0$ such that $x[-\varepsilon,0)\subset \mathring{N}$ and for every $x\in\partial N-N^o$ there exists $\delta>0$ such that $x(0,\delta]\subset\mathring{N}$.
\end{enumerate} 

These blocks form a neighborhood basis of $K$ in $M$. If the flow is {diffe\-ren\-tia\-ble}, the isolating blocks can be chosen to be differentiable manifolds which contain $N^i$ and $N^o$ as submanifolds of their boundaries and such that $\partial N^i=\partial N^o=N^i\cap N^o$. In particular, for flows defined on $\mathbb{R}^2$, the exit set $N^o$ is a disjoint union of a finite number of intervals $J_1,\ldots,J_m $ and circumpherences $C_1,\ldots,C_n$ and the same is true for the entrance set $N^i$.

Given an isolating block $N$ of an isolated invariant set $K$, its \emph{positively invariant part} is the set $N^+=\{x\in N\mid \gamma^+(x)\subset N\}$. In an analogous way, the \emph{negatively invariant part} of $N$, is the set $N^-=\{x\in N\mid \gamma^-(x)\subset N\}$. It is clear that  $N^+$ and $N^-$ are, respectively, positively and negatively invariant closed subsets of $N$ (and hence compact) whose intersection is $K$.

Let $K$ be an isolated invariant set, its \emph{Conley index} $h^+(K)$ is defined as the pointed homotopy type of the topological space $(N/N^o,[N^o])$. It is also possible to define its \emph{negative Conley index} $h^-(K)$, which agrees with the Conley index of $K$ for the reverse flow, as the pointed homotopy type of the topological space $(N/N^i,[N^i])$. A weak version of the Conley index which will be useful for us is the \emph{cohomological index} defined as $CH_+^*(K)=\check{H}^*(h^+(K))$.  The \emph{negative cohomological index} $CH_-^*(K)$ is defined in an analogous fashion. Observe that, by the strong excision property of \v Cech cohomology, $CH_+^*(K)\cong\check{H}^*(N, N^o)$ and $CH_-^*(K)\cong\check{H}^*(N,N^i)$. Our main references for the Conley index theory are \cite{Conley, Conley-Zehnder, Salamon}. We also recommend the survey \cite{Izydorek} where some connections with classical Morse theory and Brouwer degree are stated.    

We will deal through the paper with a special class of invariant sets, the so-called \emph{non-saddle} sets \cite{Bhatia}. A compact invariant set $K$ is said to be  \emph{saddle} if it admits a neighborhood $U$ such that every neighborhood $V$ of $K$ contains a point $x\in V$ with $\gamma^+(x)\nsubseteq U$ and $\gamma^-(x)\nsubseteq U$. Otherwise we say that $K$ is \emph{non-saddle}. For instance, attractors (asymptotically stable sets), repellers (negatively asymptotically stable sets) and unstable attractors with no external explosions (see \cite{Athanassopoulos,Moron,Sanchez-Gabites Transactions}) are non-saddle sets. 

If $K$ is an invariant set, its stable manifold $\mathcal{A}(K)$ is the set of all points $x\in M$ such that $\emptyset\neq\omega(x)\subset K$. Similarly, the unstable manifold $\mathcal{R}(K)$ is the set of all points $x\in M$ such that $\emptyset\neq\omega^*(x)\subset K$ . The region of influence of an invariant set $K$ is the set $\mathcal{I}(K)=\mathcal{A}(K)\cup\mathcal{R}(K)$. 

An isolated non-saddle set is said to be \emph{simple} if $\mathcal{A}(K)\cap\mathcal{R}(K)=K$. It is well known that the inclusion $i:K\hookrightarrow\mathcal{A}(K)$ of an attractor in its basin of attraction is a shape equivalence, and the same is true for the inclusion $j:K'\hookrightarrow\mathcal{R}(K')$ of a repeller in its basin of repulsion \cite{Kapitanski,Sanjurjo LN}. By combining these two facts it is easy to see that if $K$ is simple, the inclusion $i:K\hookrightarrow \mathcal{I}(K)$ is a shape equivalence.

We will also make use of a classical result of C. Guti\'{e}rrez about smoothing of $2$-dimensional flows.

\begin{theorem}[Guti\'{e}rrez \cite{Gutierrez}] Let $\varphi :M\times \mathbb{R}\rightarrow M$ be a continuous flow on a compact $C^{\infty }$ 2-manifold $M$. Then there
exists a $C^{1}$ flow $\psi $ on $M$ which is topologically equivalent to $\varphi $. Furthermore, the following conditions are equivalent:
\begin{enumerate}
\item any minimal set of $\varphi $ is trivial;

\item $\varphi $ is topologically equivalent to a $C^{2}$ flow;

\item $\varphi $ is topologically equivalent to a $C^{\infty }$ flow.
\end{enumerate}
\end{theorem}

By a trivial minimal set we understand a fixed point, a closed trajectory or the whole manifold if $M$ is the $2$-dimensional torus and $\varphi$ is topologically equivalent to an irrational flow. We readily deduce from Guti\'{e}rrez' Theorem applied to the Alexandrov compactification of the plane that continuous flows $\varphi :\mathbb{R}^{2}\times \mathbb{R}\rightarrow \mathbb{R}^{2}$ are topologically equivalent to $C^{\infty }$ flows.

We use some basic results about planar vector fields through the paper. Two good {re\-fe\-ren\-ces} covering this material are the book of Hirsch, Smale and Devaney \cite{Hirsch} and the monograph of Palis and de Melo \cite{Palis}. 

Notice that to avoid trivial cases when we consider an (isolated) invariant set $K$ it will be implicit that $\emptyset\neq K\subsetneq M$ unless otherwise specified.

 The authors are grateful to J.M. Montesinos-Amilibia and J.J. S\'{a}nchez-Gabites for useful and inspiring conversations.

\section{Topological aspects of non-saddle sets}
In this section we will study some topological aspects of a flow $\varphi: M\times\mathbb{R}\to M$ defined on a locally compact metric space having an isolated non-saddle set. For instance, we will give a characterization of non-saddleness in terms of influence-like properties and we will also characterize the shape of those non-separating isolating non-saddle sets which are neither attractors nor repellers in the torus. 

We start by stating a well-known result about isolating blocks of isolated non-saddle sets whose proof we include here for the sake of completeness.

\begin{proposition}\label{neighborhood}
Every isolated non-saddle set admits a basis of isolating blocks of the form $N=N^+\cup N^-$.
\end{proposition}

\begin{proof}
Since any isolated invariant compactum admits a neighborhood basis comprised of isolating blocks, it would be sufficient to prove that given an arbitrary isolating block $N$ of an isolated non-saddle set $K$, it contains an isolating block $N_0=N_0^+\cup N_0^-$.

 Let $N$ be an isolating block of the isolated non-saddle set $K$. Then, $K$ being non-saddle, there exists a neighborhood $V\subset N$ of $K$ such that for each $x\in V$, either $\gamma^+(x)\subset N$ or $\gamma^-(x)\subset N$. As a consequence, the compact subset $N_0=N^+\cup N^-$ is an isolating neighborhood of $K$ since $V\subset N_0$ and it is contained in the isolating block $N$. Moreover, $N_0$ is an isolating block. To see this we prove that $\partial N_0\subset\partial N$. Suppose that there exists $x\in\partial N_0-\partial N$. Then, $x\in\mathring{N}$ and there exists a sequence $x_n$ in $N-(N^+\cup N^-)$ such that $x_n\to x$. Suppose that $x\in N^+$, otherwise the argument is analogous. Let $t_n$ be a sequence of positive times , $t_n\to\infty$. Then, $xt_n\to K$ as $n\to\infty$ and, maybe after passing to a subsequence, $x_nt_n\to K$ as $n\to\infty$. Besides, the choice of $x_n$ ensures that neither $\gamma^+(x_nt_n)$ nor $\gamma^-(x_nt_n)$ are contained in $N$ leading to a contradiction with $K$ being non-saddle. Indeed, $x_nt_n$ cannot be in $N^-$ since $x_n$ leaves $N$ in negative time, being $x_n\in N-(N^+\cup N^-)$ and, since $t_n>0$ so does $x_nt_n$. We see that $x_nt_n$ cannot be in $N^+$. Suppose, to get a contradiction, that $x_nt_n\in N^+$ for almost all $n$. Then, there exists a sequence $0<s_n<t_n$ such that $x_ns_n\in N^o$ for all $n$. Moreover, the sequence $s_n$ must be bounded. If not, $x_ns_n$ would have a subsequence convergent to a point $z\in N^-$ and, as a consequence, $x_n$ would have a subsequence such that the positive semi-trajectory of each one of its elements gets arbitrarily close to $K$ before leaving $N$ in contradiction with its non-saddleness. Thus, we may assume that $s_n\to s_0\geq 0$ and, hence, $x_ns_n\to xs_0$. Since $x_ns_n\in N^o$ for each $n$, so does $xs_0$. However, $xs_0\in N^+$ which has empty intersection with $N^o$ and we get a contradiction. Therefore, $\partial N_0$ agrees with $\partial N\cap N_0$ and the compact subsets $N_0^i=N^+\cap\partial N$ and $N_0^o=N^-\cap\partial N$ are respectively an entrance and an exit set ensuring that $N_0$ is an isolating block. 
\end{proof}

\begin{remark}
It is not difficult to see that the isolating block $N_0$ defined in the proof of Proposition~\ref{neighborhood} agrees with the union of those components of $N$ which contain some component of $K$. It follows from this fact that every connected isolating block of a connected non-saddle set is of the form $N^+\cup N^-$.  
\end{remark}

A nice consequence of Proposition~\ref{neighborhood} is that for flows defined on locally compact ANRs, isolated non-saddle sets have the shape of finite polyhedra and hence, finitely generated \v Cech homology and cohomology. Moreover, it can be seen that if $N=N^+\cup N^-$, the inclusion $i:K\hookrightarrow N$ is a shape equivalence. These results were obtained in \cite{GMRSD}. 

We see that, in analogy with the basin of attraction of an attracting set (attractors and unstable attractors), the region of influence of an isolated non-saddle set is an open set. 

\begin{proposition}
If $K$ is an isolated non-saddle compactum then $\mathcal{I}(K)$ is an open neighborhood of $K$.
\end{proposition}

\begin{proof}
Given an isolating neighborhood of $K$ of the form $N=N^{+}\cup N^{-}$, it must be contained in $\mathcal{I}(K),$ hence $\mathcal{I}(K)$ is a neighborhood of $K$. On the other hand, if $x\in \mathcal{I}(K)$ and, say, $\emptyset\neq\omega(x)\subset K $, then there is a neighborhood $U$ of $x$ and a $t_{0}\geq 0$
such that $Ut_{0}$ is contained in $N^{+}$. Hence for every $y\in Ut_{0}$ we have that $\emptyset\neq\omega (y)\subset K$ and, as a consequence, the same thing
happens for every $z\in U.$ Therefore $\mathcal{I}(K)$ is open.
\end{proof}

However, in contrast with the case of attracting sets, the converse does not {ne\-ce\-ssa\-ri\-ly} hold.

\begin{remark}\label{saddle}
There are isolated saddle continua $K$ such that $\mathcal{I}(K)$ is an open neighborhood of $K$, hence this property does not characterize
non-saddleness. For instance, consider the Mendelson's flow on the plane \cite{Mendelson}, see Figure~\ref{fig:Mendel}.
\begin{figure}[h]
\center
\includegraphics[scale=0.3]{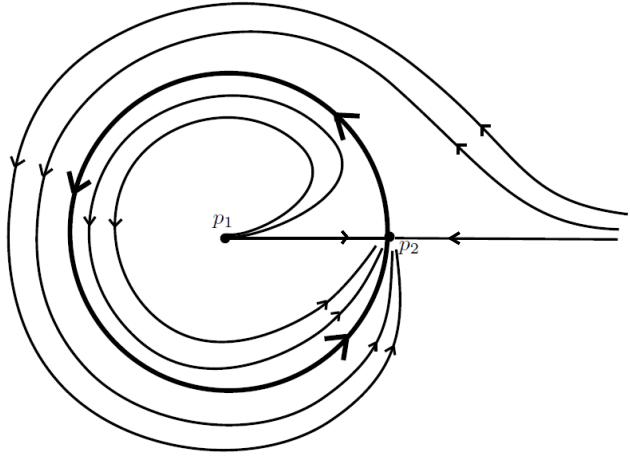}
\caption{Mendelson's attractor}
\label{fig:Mendel}
\end{figure}

Then $K=\{p_{2}\}$ is a saddle set (in fact, an unstable attractor) and its region of influence (region of attraction in this case)
is $\mathbb{R}^{2}-\{p_{1}\}.$
\end{remark}

The next result gives a sufficient condition for a non-saddle set to be either an attractor or a repeller.

\begin{proposition}\label{small}
Let $K$ be an isolated non-saddle continuum of a flow on a locally compact metric space $M$. Suppose that $K$ has arbitrarily small neighborhoods in 
$M$ which are not disconnected by $K$. Then $K$ is either an attractor or a repeller. In particular, if $M$ is an orientable $n$-dimensional manifold with $n>1$
and $K$ has trivial shape or, more generally, if $\check{H}^{n-1}(K)=\{0\}$, then $K$ is either an attractor or a repeller.
\end{proposition}

\begin{proof}
Consider a connected isolating block $N$ of $K$ of the form $N=N^{+}\cup N^{-}$, which is contained in $\mathcal{I}(K)$, and let $W$ be a neighborhood of $K$
such that $W\subset N$ and $W$ is not disconnected by $K$. Then either $W\cap (N^{+}-K)$ or $W\cap (N^{-}-K)$ is empty. In the first case $K$ is a
repeller and in the second case an attractor. If $M$ is an orientable $n$-manifold then, by Alexander duality, $H_{1}(M,M-K)\cong \check{H}^{n-1}(K)$ and, by
excision, $H_{1}(\mathring{N},\mathring{N}-K)\cong H_{1}(M,M-K)$. Thus, if $K$ has trivial shape or more generally $\check{H}^{n-1}(K)=\{0\}$ from the terminal part
of the exact homology sequence of the pair $(\mathring{N},\mathring{N}-K)$ 
\[
\ldots\rightarrow H_{1}(\mathring{N},\mathring{N}-K)=\{0\}\rightarrow \widetilde{H}_{0}(\mathring{N}-K)\rightarrow 
\widetilde{H}_{0}(\mathring{N})=\{0\} 
\]
we get that $\widetilde{H}_{0}(\mathring{N}-K)=\{0\}$. As a consequence $\mathring{N}$ is not disconnected by $K$. Since $N$ can be taken arbitrarily small, this is a particular case of the situation considered before.
\end{proof}

Throughout this paper we will often use the torus as a relevant phase space for flows which illustrate the main notions we introduce. The previous
proposition, together with some classical results in Algebraic Topology, can be used to describe the topological structure of an important class of
non-saddle sets in the torus.

\begin{theorem}
Suppose that $K$ is an isolated non-saddle continuum of a flow on the torus, $T$, such that $K$ does not separate $T$ and $K$ is neither an attractor nor
a repeller. Then $K$ has the shape of a circle. Moreover, if $K$ does not contain fixed points it is either a limit cycle or homeomorphic to a closed annulus bounded by two limit cycles.
\end{theorem}

\begin{proof}
We take coefficients in $\mathbb{Z}_{2}.$ Since $K$ does not separate $T$ we have that $\widetilde{H}_{0}(T-K)=\{0\}$. By using the exact homology sequence
of the pair $(T,T-K)$ 
\[
\ldots\rightarrow H_{1}(T)\rightarrow H_{1}(T,T-K)\rightarrow \widetilde{H}
_{0}(T-K)=\{0\}\rightarrow\ldots
\]%
and Alexander duality we get that $\rank\check{H}^{1}(K)=\rank H_{1}(T,T-K)\leq 2$. We will prove, arguing by contradiction, that, in
fact, $\rank\check{H}^{1}(K)\neq 2$. Suppose $\rank\check{H}^1(K)= 2$ and let $N=N^+\cup N^-$ be an isolating block of $K$.  Consider the terminal part of the reduced homology long exact sequence of the pair $(T,N)$ 
\begin{align*}
 \{0\}=H_{2}(N)\rightarrow H_{2}(T)=\mathbb{Z}_{2}\rightarrow
H_{2}(T,N)\rightarrow H_{1}(N)\rightarrow H_{1}(T)\to \\
H_1(T,N)\to \widetilde{H}_0(N)=\{0\}
\end{align*}
The homomorphism $H_{2}(T)\rightarrow H_{2}(T,N)$ is an isomorphism since $ H_{2}(T,N)=H^{2}(T,N)$ by the Universal Coefficients Theorem and by
Alexander duality $H^{2}(T,N)=\check{H}^{2}(T,K)\cong H_{0}(T-K)$, which is $\mathbb{Z}_{2}$ since $K$ does not separate $T$. As a consequence, the homomorphism $H_1(N)\to H_1(T)$ is injective and hence an isomorphism. Therefore, $H_1(T,N)=\{0\}$ and by excising $K$, $H_1(T-K,N-K)=\{0\}$. Then, from the homology long exact sequence of the pair $(T-K,N-K)$ we get $H_0(N-K)\cong H_0(T-K)=\mathbb{Z}_2$, i.e. $K$ does not separate $N$. Besides, $N$ can be chosen arbitrarily small, $N$ being an {iso\-la\-ting} block, and hence Proposition~\ref{small} ensures that it has to be an attractor or a repeller in contradiction with the assumption. As a consequence, $\rank \check{H}^1(K)$ is either $0$ or $1$ and by \cite{Borsuk Proc, Dydak-Segal, Mcmillan1, Mcmillan2} $K$ has either trivial shape or the shape of a circle respectively. The first case is excluded by Proposition~\ref{small} and hence $K$ has the shape of a circle. 

For the last part of the statement observe that $K$, having the shape of a circle, admits an isolating block $N=N^+\cup N^-$ which is an annulus by \cite{Gutierrez} and \cite{Conley-Easton}. The annulus $N$ can be embedded into $\mathbb{R}^2$ endowed with a smooth flow which preserves the dynamics in a small neighborhood of $K$ contained in $\mathring{N}$. This can be achieved by a slight modification of \cite[Example~21]{Sanchez-Gabites Transactions} combined with Guti\'errez Theorem.  Then, if $K$ does not have fixed points, by \cite{Barge-Sanjurjo} the result follows. 
\end{proof}

Remark~\ref{saddle} shows that unstable attractors are not necessarily non-saddle. As a matter of fact we have the following characterization, whose proof is given
in \cite{Sanjurjo LN}.

\begin{proposition}\label{unstable-nonsaddle}
Let $K$ be an unstable attractor of a flow. Then $K$ has no external explosions if and only if it is non-saddle.
\end{proposition}

In order to characterize non-saddle sets by influence-like properties we introduce the following notion first.

\begin{definition}\label{sinfluence}
A point $p$ is \emph{strongly influenced} by a compact invariant set $K$ if it has a neighborhood $U_{p}$ with the following property: for every neighborhood $V$ of $K$ there is a $T\geq 0$ such that for every $x\in U_{p}$ we have $x[T,\infty )\subset V$ or $x(-\infty ,-T]\subset V.$ We also say that the
neighborhood $U_{p}$ is strongly influenced by $K$. There are, with obvious changes, similar definitions for the notions of a point and a
neighborhood \emph{strongly attracted} or \emph{strongly repelled} by $K$.
\end{definition}

If $p$ is strongly influenced by $K$ then $p\in \mathcal{I}(K)$. Moreover, it is clear that if $p$ is strongly influenced by $K$ then all points in $%
\gamma (p)$ are strongly influenced by $K$.

Definition~\ref{sinfluence} provides all we need to characterize non-saddleness.

\begin{proposition}
The following are equivalent for an isolated invariant compactum $K$:
\begin{enumerate}
\item[i)] $K$ is non-saddle;
\item[ii)] All points of $K$ are strongly influenced by $K$;
\item[iii)] $K$ has a neighborhood $U$ all whose points are strongly influenced by $
K$.
\end{enumerate}

Moreover, if $K$ is non-saddle then $\mathcal{I}(K)$ agrees with the set of all points strongly influenced by $K$, all points in $\mathcal{A}(K)-K$
are strongly attracted by $K$ and all points in $\mathcal{R}(K)-K$ are strongly repelled by $K$.
\end{proposition}

\begin{proof}
Conditions ii) and iii) are clearly equivalent since, by definition, strong influence on a point $p$ requires strong influence on all points of a
neighborhood $U_{p}$. Moreover, if $K$ is non-saddle then every point in the interior of an isolating neighborhood of $K$ of the form $N=N^{+}\cup N^{-}$
is strongly influenced by $K$ and, as a consequence, we have that i) implies iii). On the other hand suppose that all points of $K$ are strongly
influenced by $K.$ We claim that for every isolating neighborhood $N$ of $K$ there is an $\epsilon >0$ such that if a point $x\in N$ abandons $N$ in the
past and in the future then $d(x,K)>\epsilon $. If not, there is and isolating neighborhood $N$ and a sequence of points $x_{n}$ contained in $N$
\ with $x_{n}\rightarrow x\in K$ such that every $x_{n}$ abandons $N$ in the past and in the future in times, say $T_{n}<0$ and $T_{n}^{\prime }>0.$
Since $x_{n}\rightarrow x\in K$ and $K$ is invariant we must have that $T_{n}\rightarrow -\infty $ and $T_{n}^{\prime }\rightarrow \infty .$
However, this is in contradiction with the fact that $x$ is strongly influenced by $K$. Hence such an $\epsilon >0$ exists. As a consequence, if
we define $N_{0}=N^{+}\cup N^{-}$ we obtain another isolating neighborhood $N_{0}\subset N$ with $N_{0}^{+}=N^{+}$ and $N_{0}^{-}=N^{-}$ and $K$ is
non-saddle. Thus ii) implies i). Concerning the last assertion in the statement of the proposition, if $K$ is non-saddle and $x\in $\ $\mathcal{I}%
(K)$ then $\gamma (x)$ enters every isolating neighborhood of the form $N=N^{+}\cup N^{-}$, all whose points are strongly influenced by $K.$ Hence $%
x $ is strongly influenced by $K$. If $x\in \mathcal{A}(K)-K$ then $\gamma^{+}(x)$ enters $N^{+}$ and, hence $x$ is strongly attracted and similarly,
if $x\in \mathcal{R}(K)-K$ then $x$ is strongly repelled.
\end{proof}

\section{On the structure of a flow having a non-saddle set}

This section is devoted to the study of the structure of a flow on a locally compact metric space having a non-saddle set. For this purpose we will use the following notation:

\begin{enumerate}
\item[-] $\mathcal{H}(K)=\mathcal{A}(K)\cap \mathcal{R(}K)$, the set of all points $x$ such that $\omega (x)\neq \emptyset ,$ $\omega^* (x)\neq \emptyset $ and $\omega(x)\cup \omega^* (x)\subset K$. If $x\in \mathcal{H}(K)-K$ we say that the point $x$ and the trajectory $\gamma (x)$ are \emph{homoclinic}.

\item[-] $\mathcal{A}^*(K)=\mathcal{A}(K)-\mathcal{R(}K)$, the set of all points $x$ such that $\emptyset \neq \omega (x)\subset K$ but $\omega^*
(x)\nsubseteq K$ or $\omega^*(x)=\emptyset$.

\item[-] $\mathcal{R}^*(K)=\mathcal{R}(K)-\mathcal{A(}K)$, the set of all points $x$ such that $\emptyset \neq \omega^* (x)\subset K$ but $\omega
(x)\nsubseteq K$ or $\omega (x)=\emptyset $.
\end{enumerate}

\begin{proposition}\label{topol}
Let $K$ be an isolated non-saddle set. Then:

i) $\mathcal{H}(K)-K$ is an open set in $M$.

ii) $\mathcal{A}^{\ast }(K)\cup K$ and $\mathcal{R}^{\ast }(K)\cup K$ are
closed in $\mathcal{I}(K)$.
\end{proposition}

\begin{proof}
If $x\in $\ $\mathcal{H}(K)-K$ then $x$ is both strongly attracted and strongly repelled by $K,$ which means that it has a neighborhood $U_{x}$
contained in $\mathcal{A}(K)\cap \mathcal{R(}K)$ and not meeting $K.$ Hence $\mathcal{H}(K)-K$ is open.

To prove ii) let us argue by contradiction. If $\mathcal{A}^*(K)\cup K$ is not closed in $\mathcal{I}(K)$ then there exists a sequence $%
x_{n}\rightarrow x\in \mathcal{I}(K)$ with $x_{n}\in \mathcal{A}^*(K)\cup K$ and $x\notin \mathcal{A}^{\ast }(K)\cup K.$ Since $x\notin 
\mathcal{A}^{\ast }(K)\cup K$ we must have that $x\in \mathcal{R(}K)-K$ and, as a consequence, $x$ is strongly repelled by $K$. But this implies that $%
x_{n}$ is repelled by $K$ for almost all $n$, which is in contradiction with the choice of the sequence $x_{n}$.
\end{proof}

In the sequel we will be concerned with the study of the region of influence $\mathcal{I}(K)$ and, in particular, with the structure of $%
\mathcal{I}(K)-K$. By the previous results, $\mathcal{I}(K)-K$ is the disjoint union of the sets $\mathcal{H}(K)-K,$ $\mathcal{A}^{\ast }(K)$ and $%
\mathcal{R}^{\ast }(K)$ where $\mathcal{H}(K)-K$ is open and $\mathcal{A}^*(K)$ and $\mathcal{R}^{\ast }(K)$ are closed in $\mathcal{I}(K)-K$.
However, $\mathcal{H}(K)-K$ is not necessarily closed in $\mathcal{I}(K)-K$ as the following example shows.

\begin{example}\label{example1}
This example shows an isolated non-saddle set $K$ of a flow on the torus $T$ (which is represented in Figure~\ref{fig:torus flow} as a square with opposite sides identified). 
\begin{figure}[h]
\center
\begin{pspicture}(-3,-2)(4,3)
\psline(-2,-2)(-2,3)
\psline(3,-2)(3,3)

\psline{*->}(-2,3)(0.5,3)	
\psline{-*}(0.4,3)(3,3)

%\psline{*->}(-2,2.75)(0.5,2.75)	
%\psline{-*}(0.4,2.75)(3,2.75)

\psline{*->}(-2,2.5)(0.5,2.5)	
\psline{-*}(0.4,2.5)(3,2.5)

%\psline{*->}(-2,2.25)(0.5,2.25)	
%\psline{-*}(0.4,2.25)(3,2.25)

\psline{*->}(-2,2)(0.5,2)	
\psline{-*}(0.4,2)(3,2)

%\psline{*->}(-2,1.75)(0.5,1.75)	
%\psline{-*}(0.4,1.75)(3,1.75)

\psline{*->}(-2,1.5)(0.5,1.5)	
\psline{-*}(0.4,1.5)(3,1.5)

%\psline{*->}(-2,1.25)(0.5,1.25)	
%\psline{-*}(0.4,1.25)(3,1.25)

\psline{*->}(-2,1)(0.5,1)	
\psline{-*}(0.4,1)(3,1)

%\psline{*->}(-2,0.75)(0.5,0.75)	
%\psline{-*}(0.4,0.75)(3,0.75)

\psline[linecolor=red]{->}(-2,0.5)(-0.75,0.5)
\psline[linecolor=red]{-*}(-0.78,0.5)(0.4,0.5)
\psline[linecolor=red]{->}(0.4,0.5)(1.75,0.5)
\psline[linecolor=red](1.72,0.5)(3,0.5)
\psdot(-2,0.5)
\psdot(3,0.5)

%\psline{*->}(-2,0.25)(0.5,0.25)	
%\psline{-*}(0.4,0.25)(3,0.25)

\psline{*->}(-2,0)(0.5,0)	
\psline{-*}(0.4,0)(3,0)

%\psline{*->}(-2,-0.25)(0.5,-0.25)	
%\psline{-*}(0.4,-0.25)(3,-0.25)

\psline{*->}(-2,-0.5)(0.5,-0.5)	
\psline{-*}(0.4,-0.5)(3,-0.5)

%\psline{*->}(-2,-0.75)(0.5,-0.75)	
%\psline{-*}(0.4,-0.75)(3,-0.75)

\psline{*->}(-2,-1)(0.5,-1)	
\psline{-*}(0.4,-1)(3,-1)

%\psline{*->}(-2,-1.25)(0.5,-1.25)	
%\psline{-*}(0.4,-1.25)(3,-1.25)

\psline{*->}(-2,-1.5)(0.5,-1.5)	
\psline{-*}(0.4,-1.5)(3,-1.5)

%\psline{*->}(-2,-1.75)(0.5,-1.75)	
%\psline{-*}(0.4,-1.75)(3,-1.75)

\psline{*->}(-2,-2)(0.5,-2)	
\psline{-*}(0.4,-2)(3,-2)

\rput(0.55,0.30){$p$}
\rput(-2.3,1.75){$K$}
\rput(3.3,-0.75){$K$}
\end{pspicture}
\caption{Flow on the torus}
\label{fig:torus flow}
\end{figure}
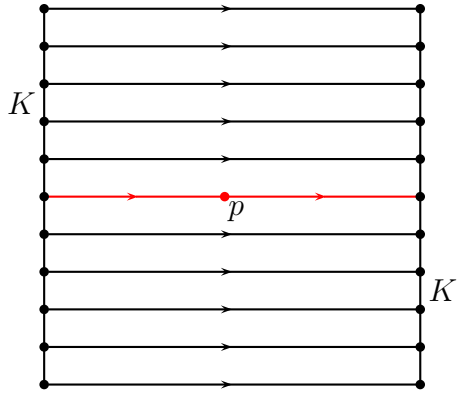
All the points of $K$ are stationary and, in addition, there is a fixed point $p\notin K$. The orbits of all points $x\in T-K$ are homoclinic except
the equilibrium $p$, the orbit finishing in $p$ and the orbit starting in $p$. The region of influence of $K$ is $\mathcal{I}(K)=T-\{p\}.$ The set $%
\mathcal{H}(K)-K$ is not closed in $\mathcal{I}(K)-K.$
\end{example}

The fact that $\mathcal{H}(K)-K$ is not necessarily closed in $\mathcal{I}(K)-K$ (and, hence, $\mathcal{H}(K)$ is not necessarily closed in $\mathcal{I}(K)$) accounts for much of the complexity of the structure of $\mathcal{I}(K)$, specially when compared, for example, with the more simple case of
unstable attractors without external explosions, where $\mathcal{H}(K)$ is indeed closed in $\mathcal{I}(K)=\mathcal{A}(K)$. In spite of this, some
properties of that family of unstable attractors are shared by non-saddle sets. However, if we want to have some understanding of the structure of the
region of influence of non-saddle sets we must acknowledge the existence of a special kind of orbits which are responsible for both the topological and
the dynamical complexity. This we will do in Definition~\ref{dpoint} .

In the sequel we will often use the prolongational limits $J^{+}$, $J^{-}$ and $J^*$. The following result, whose proof is left to the reader,
provides a useful {cha\-rac\-te\-ri\-za\-tion} of attracting, repelling and homoclinic points lying outside $K$.

\begin{proposition}
Let $K$ be an isolated non-saddle set of a flow on $M$ and $x\in M-K$. Then
\begin{enumerate}
\item[i)] $x\in \mathcal{A}(K)$ if and only if $J^{+}(x)\neq \emptyset $ and $
J^{+}(x)\subset K.$

\item[ii)] $x\in \mathcal{R}(K)$ if and only if $J^{-}(x)\neq \emptyset $ and $%
J^{-}(x)\subset K.$

\item[iii)] If $x\in \mathcal{H}(K)$ then $J^{\ast }(x)\neq \emptyset $ and $J^{\ast }(x)\subset K\times K.$ The converse holds if $M$ is compact.
\end{enumerate}
\end{proposition}

We stress that the previous proposition refers to points $x\in M-K$ only$.$ These properties do not generally hold for points in $K.$

In the following definition we introduce a kind of points which play an {essen\-tial} role in our discussion.

\begin{definition}\label{dpoint}
A point $x\in \mathcal{I}(K)$ is said to be \emph{positively dissonant} if $x\notin\mathcal{A}(K)$ (in which case $x\in \mathcal{R}(K)$) but $J^{+}(x)\cap
K\neq \emptyset $. We also say that the orbit $\gamma (x)$ is positively dissonant. There is a similar definition for \emph{negatively dissonant} points and
orbits. A point $x\notin \mathcal{I}(K)$ and its orbit are said to be \emph{externally dissonant} if $J^{\ast }(x)\cap (K\times K)\neq \emptyset .$ We
denote by $\mathcal{D}$ the set of all dissonant points .
\end{definition}

This definition conveys the idea that positively dissonant points are not attracted by $K$ but, nevertheless, $K$ has a kind of attractive influence
on some points close to them. Externally dissonant points do not belong to the region of influence of $K$ (therefore they are neither atracted nor
repelled) but $K$ has \emph{simultaneously} a kind of attractive and repulsive influence on some points close to them. We remark that a flow on $M$ migth
have an isolated non-saddle set $K$ and points $x$ in $M-\mathcal{I}(K)$ with $J^{+}(x)\cap K\neq \emptyset $ and $J^{-}(x)\cap K\neq \emptyset $ but $%
J^{\ast }(x)\cap (K\times K)=\emptyset .$ Obviously, such points are not externally dissonant.

By using dissonant points we can give a nice dynamical characterization of the closure of the set of homoclinic points.

\begin{proposition}\label{closure}
Let $x$ be a point not contained in $K$. Then $x$ is dissonant if and only if it is in the boundary of $\mathcal{H}(K)$. As a consequence $\overline{\mathcal{H}(K)}=\mathcal{H}(K)\cup\mathcal{D}$ i.e. the closure or $\mathcal{H}(K)$ consists of $K$ and its homoclinic points together with the dissonant points of $K$.
\end{proposition}

\begin{proof}
Suppose $x$ is in the boundary of $\mathcal{H}(K)$. Then $x$ is the limit of a sequence of points $x_n\in\mathcal{H}(K)$, and hence there exist a subsequence $x_{n_k}$, sequences $t_k\to\infty$ and $s_k\to -\infty$ and points $y,z\in K$ with $x_{n_k}t_k\to y$, $x_{n_k}s_k\to z$. As a consequence, if $x\notin\mathcal{I}(K)$ then $x$ is an externally dissonant point. If $x\in \mathcal{I}(K)-K$  and $x\in\mathcal{A}(K)$ then $x\notin\mathcal{R}(K)$ since, otherwise, $x$ would be in $\mathcal{H}(K)-K$, which is an open set of $M$, and this is in contradiction with $x$ being in the boundary of $\mathcal{H}(K)$. Since $x\notin\mathcal{R}(K)$, $x_{n_k}\to x$ and $x_{n_k}s_k\to z\in K$ we have that $x$ is a negatively dissonant point. An analogous argument applies when $x\in\mathcal{R}(K)$. This proves that $x$ is dissonant whenever $x$ is in the boundary of $\mathcal{H}(K)$ and $x\notin K$. On the other hand if $x\notin \mathcal{I}(K)$ is an externally dissonant point, there exist $x_n\to x$, $t_n\to\infty$, $s_n\to -\infty$ with $x_nt_n\to y$ and $x_ns_n\to z$ and $(y,z)\in K\times K$. Then the points $x_n\in\mathcal{H}(K)$ for all $n$ sufficiently large and, thus, $x$ is in the boundary of $\mathcal{H}(K)$. Positively and negatively dissonant points are readily seen to belong to the boundary of $\mathcal{H}(K)$.      
\end{proof}

\begin{remark}
Observe that if $K$ does not have dissonant points and $\mathcal{H}(K)-K\neq\emptyset$, then $\mathcal{H}(K)$ is compact and hence an isolated non-saddle set. Moreover, $\mathcal{H}(K)$ is the smallest simple non-saddle set containing $K$ and its region of influence agrees with $\mathcal{I}(K)$.
\end{remark}

The next result deepens into the topological structure of $\mathcal{I}(K)-K$.

\begin{proposition}\label{topological structure}
Let $K$ be an isolated non-saddle set of a flow $\varphi$ defined on a locally compact ANR $M$ and $N=N^{+}\cup N^{-}$ an isolating block of $K$ contained in $\mathcal{I}(K)$. Then, $\mathcal{I}(K)-K$ has a finite number of connected components. Moreover, if a component $C$ does not contain dissonant points the restriction flow $\varphi|C$ is parallelizable with section a component of $\partial N$. Hence, the flow provides a strong deformation retraction of $C$ onto a component of $\partial N\cap C$. Conversely, if the restriction flow $\varphi|C$ to a component $C$ of $\mathcal{I}(K)-K$ is parallelizable, then $C$ does not contain dissonant points.  
\end{proposition}
 
\begin{proof}
It is easy to see that $N$ lies in the interior of another isolating block $N_{0}=N_{0}^{+}\cup N_{0}^{-}$ and that there exists a retraction of $\mathring{N}_0-K$ onto $\partial N$ provided by the flow. As a consequence, $\partial N$ is an ANR and hence, it has a finite number of components. Since every component of $\mathcal{I}(K)-K$ contains at least one component of $\partial N$ there must be a finite number of them. 

Suppose that $C$ is a component of $\mathcal{I}(K)-K$ which does not contain dissonant points. Then, by Proposition~\ref{topol} and Proposition~\ref{closure},  $C$ is contained in one of the subsets $\mathcal{A}^*(K)$, $\mathcal{R}^*(K)$ or $\mathcal{H}(K)-K$ and, hence, $K$ is either a global attractor or a global repeller or a global unstable attractor with no external explosions respectively for the restriction flow $\varphi|C\cup K$. Then, by \cite{Bhatia-Szego} and \cite{Sanchez-Gabites Transactions} the flow in $C$ is parallelizable with section a component of $\partial N$.

Conversely, assume that $\varphi|C$ is parallelizable. We see that $C$ does not contain dissonant points. Choose an isolating block $N=N^+\cup N^-$ of $K$. Suppose, arguing by contradiction, that $x\in C$ is a positively dissonant point the other case being analogous. Then there exists a sequence of homoclinic points $x_n\in C$ such that $x_n\to x$. For each $n$ choose $t_n$ to be the unique time for which $x_nt_n\in N^i\cap C$. Since $N^i\cap C$ is compact we may assume that $x_nt_n\to y\in N^i\cap C$. Besides, since $x$ is positively dissonant then $x\notin\mathcal{A}(K)$ and, as a consequence, $t_n\to +\infty$ in contradiction with $\varphi|C$ being parallelizable. 
\end{proof}

\begin{remark}
For the second part of Proposition~\ref{topological structure} there is no need of $M$ to be an ANR. It also holds if $M$ is a locally compact metric space.
\end{remark}

The existence of positively, negatively and externally dissonant points is mutually related as the following result shows.

\begin{proposition}\label{three}
If an isolated non-saddle set $K$ has externally dissonant points, then it has also positively and negatively dissonant points. Conversely, if $M$ is compact and $K$ has either positively or negatively dissonant points in $\mathcal{I}(K)$ then $K$ has externally dissonant points.
\end{proposition}
\begin{proof}
Suppose that $x$ is an externally dissonant point. Then there exists a sequence $x_n$ in $\mathcal{H}(K)-K$ such that $x_n\to x$. Let $N=N^+\cup N^-$ be an isolating block of $K$ (which must be contained in $\mathcal{I}(K)$). For almost all $n$ there exists $t_n>0$ such that $x_nt_n\in N^i$. By compactness we may assume that $x_nt_n\to y\in N^i\subset\mathcal{A}(K)-K$. The sequence $t_n\to+\infty$ since, otherwise it would have a bounded subsequence $t_{n_k}$ such that $t_{n_k}\to t_0$. Then, $x_{n_k}t_{n_k}\to y=xt_0$ which is a contradiction since $xt_0\notin\mathcal{I}(K)$. As a consequence $y\in J^+(x)$ and by \cite{Bhatia-Szego} $x\in J^-(y)$. Then $J^-(y)\cap K\neq\emptyset$ but $J^-(y)\nsubseteq K$ and, hence, $y$ is a negatively dissonant point. An analogous argument leads us to find a positively dissonant point.

For the converse, suppose that $M$ is compact and there exists a negatively dissonant point $x$. Then there exists a sequence $x_n$ in $\mathcal{H}(K)-K$ such that $x_n\to x$. By assumption, $\emptyset\neq\omega^*(x)\subset M-\mathcal{I}(K)$. Let $y\in\omega^*(x)$. Then there exists a sequence $t_n\to-\infty$ such that $xt_n\to y$. Given $\epsilon>0$, it is possible to choose subsequences $x_{n_k}$ and $t_{n_k}$ such that $d(x_{n_k}t_{n_k},xt_{n_k})<\epsilon/2$ for $k$ larger than a certain $k_0$. Moreover, there exists $k_1$ such that if $k\geq k_1$, then $d(xt_{n_k},y)<\epsilon/2$. As a consequence $d(x_{n_k}t_{n_k},y)\leq d(x_{n_k}t_{n_k},xt_{n_k})+d(xt_{n_k},y)<\epsilon$, for $k\geq\max\{k_0,k_1\}$. This proves that the point $y$ is externally dissonant. If the point $x$ is chosen to be positively dissonant the argument is analogous.
\end{proof}

\begin{remark} 
In the previous proposition it is proved that points $x\notin\mathcal{I}(K)$ which are $\omega^*$-limits of negatively dissonant points
are externally dissonant, and the same is true for $\omega$-limits of positively dissonant points, but it can be easily shown that the converse
does not hold. In fact, there are externally dissonant points which are neither $\omega $-limits nor $\omega ^{\ast }$-limits of points of $\mathcal{%
I}(K)$. However, in the same proposition it is proved the weaker property that externally dissonant points lie in the positive prolongational limit ($
J^{+}$) of positively dissonant points and in the negative prolongational limit ($J^{-}$) of negatively dissonant points.
\end{remark}

\begin{example}
Our previous Example~\ref{example1} can be presented in a more general way. Consider the vector field in the torus induced by the differential equation in the square $I^2=[0,1]^2\subset\mathbb{R}^2$ defined by
\begin{equation*}
\begin{cases}
\dot{x}=\psi(x)(x^2+y^2-\frac{1}{4})\\
\dot{y}=0
\end{cases}
\end{equation*}
Where $\psi:[-1,1]\to [0,1]$ is a smooth function satisfying that $\psi(x)=1$ if $x\in[-1/2,1/2]$ and $\psi(x)=0$ if and only if $x=\pm 1$.

\begin{figure}[h]
\center
\begin{pspicture}(-2,-2)(3,3)

\psline(-2,-2)(-2,3)
\psline(3,-2)(3,3)

\psline{*->}(-2,3)(0.5,3)	
\psline{-*}(0.4,3)(3,3)

\pscircle[linecolor=blue](0.5,0.5){1.25}

\psline[linecolor=red](-2,-0.75)(-0.6,-0.75)
\psline[linecolor=red]{>-*}(-0.75,-0.75)(0.5,-0.75)
\psline[linecolor=red]{->}(0.5,-0.75)(1.75,-0.75)
\psline[linecolor=red](1.70,-0.75)(3,-0.75)
\psdot(-2,-0.75)
\psdot(3,-0.75)

\psline[linecolor=red]{->}(-2,1.75)(-0.6,1.75)
\psline[linecolor=red]{-*}(-0.75,1.75)(0.5,1.75)
\psline[linecolor=red]{->}(0.5,1.75)(1.75,1.75)
\psline[linecolor=red](1.70,1.75)(3,1.75)
\psdot(-2,1.75)
\psdot(3,1.75)

%%Líneas interiores posterior en inferior%%

\psline(-0.55,1.125)(0.5,1.125)
\psline{<-}(0.4,1.125)(1.55,1.125)

\psline(-0.55,-0.125)(0.5,-0.125)
\psline{<-}(0.4,-0.125)(1.55,-0.125)
%%%%%%%%%%%%%%%%%%%%%%%%%%%%%%%%%%%%%%

\psline(-0.7,0.5)(0.5,0.5)
\psline{<-}(0.4,0.5)(1.7,0.5)

\psline{*->}(-2,0.5)(-1.2,0.5)
\psline(-1.3,0.5)(-0.76,0.5)
\psline{->}(1.76,0.5)(2.5,0.5)
\psline{-*}(2.4,0.5)(3,0.5)

\psline{*->}(-2,-1.375)(0.5,-1.375)
\psline{-*}(0.4,-1.375)(3,-1.375)

\psline{*->}(-2,-0.125)(-1,-0.125)	
\psline(-1.2,-0.125)(-0.6,-0.125)

\psline{*->}(-2,1.125)(-1,1.125)	
\psline(-1.2,1.125)(-0.6,1.125)

\psline{->}(1.6,1.125)(2.2,1.125)	
\psline{-*}(2,1.125)(3,1.125)

\psline{->}(1.6,-0.125)(2.2,-0.125)	
\psline{-*}(2,-0.125)(3,-0.125)

\psline{*->}(-2,2.375)(0.5,2.375)	
\psline{-*}(0.4,2.375)(3,2.375)

\psline{*->}(-2,-2)(0.5,-2)	
\psline{-*}(0.4,-2)(3,-2)

%%%%Texto cuadrado 2%%%%
\rput(0.65,1.9){\tiny $p$}
\rput(0.65,-0.9){\tiny $p'$}
\rput(-2.3,1.75){$K$}
\rput(3.3,-0.75){$K$}
\end{pspicture}
\caption{Saddle-node bifurcations at $p$ and $p'$}
\end{figure}
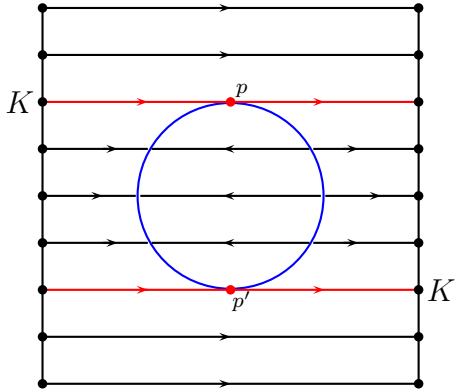

The phase diagram of this flow can be understood as a collective phase diagram of a saddle-node bifurcation at the point $p$ (and symetrically at the point $p^{\prime }$). All points contained in the blue circle are fixed. The bifurcation points agree in this example with the externally dissonant points. The boundary of the region of influence of $K$ separates the torus into two connected components.
\end{example}

We will see that the previous example is only a particular instance of a much more general situation.

S\'{a}nchez-Gabites proved in \cite{Sanchez-Gabites Transactions} a topological result which he used to {cons\-truct} families of unstable attractors in manifolds and which is useful
as well in the present context. The proof of the result makes use of a classical theorem by Thom \cite{Thom} about representation of homology classes. We
reproduce its statement here because, originally, this result is contained in the proof of a different proposition and it is not formulated separately.

\begin{proposition}\label{Gabites}
(S\'{a}nchez-Gabites). Let $M$ be a connected oriented closed smooth $n$-manifold with $H^{1}(M)\neq \{0\}$ (or, equivalently, with $H_{n-1}(M)\neq \{0\}$).
Then there exists a connected oriented closed smooth hypersurface $Z\subset M $ {sa\-tis\-fying} the following conditions:
\begin{enumerate}
\item[i)] $Z$ has a product neighborhood in $M$, i.e. there exists an open
neighborhood $U$ and a homeomorphism $h:Z\times \mathbb{R}\rightarrow U$
with $h(z,0)=z$ for every $z\in Z.$

\item[ii)] $M-U$ is connected.
\end{enumerate}
\end{proposition}

It turns out that Proposition~\ref{Gabites} is useful to construct non-saddle sets. The next result illustrates that non-saddle sets with dissonant points are
indeed abundant. Its proof is, simply, an adaptation of the argument given in \cite{Sanchez-Gabites Transactions} for unstable attractors without external explosions.

\begin{theorem}\label{existence}
Let $M$ be a connected closed oriented smooth $n$-manifold with $n\geq 2$. If $H^{1}(M)\neq \{0\}$ then there exists a flow on $M$ having a connected isolated non-saddle set
with dissonant points.
\end{theorem}

\begin{proof}
We use the notation of Proposition~\ref{Gabites}. Construct a flow on $Z\times \mathbb{R}$ such that $Z\times (-\infty,0]$ and $Z\times [1,\infty)$ consist of fixed points and points in $Z\times (0,1)$ move from $Z\times \{0\}$ to $Z\times\{1\}$ except for a distinguished point $z_{0}$ for which the interval $
\{z_{0}\}\times (-1,1)$ is broken in three different trajectories covering $\{z_{0}\}\times (-1,0),$ $\{z_{0}\}\times \{0\}$ and $\{z_{0}\}\times (0,1)$ respectively. Now carry this flow to $U$ via $h$ and extend it to $M$ by leaving fixed all points in $M-U$. It is easy to see that $K=M-U$ \ is an isolated non-saddle set, $h(\{z_{0}\}\times (-1,0))$ and $h(\{z_{0}\}\times (0,1))$ are positively and negatively dissonant orbits and $h(z_{0})$ is an externally dissonant point.
\end{proof}

Theorem~\ref{existence} applies to manifolds with $H^{1}(M)\neq \{0\}.$ However, for manifolds with $H^{1}(M)=\{0\}$ the situation is much simpler.

\begin{theorem}\label{structure}
Let $M$ be a connected closed manifold with $H^{1}(M)=\{0\}$ and suppose that $K$ is a connected isolated non-saddle set of a flow on $M$. Then, $K$ does not have dissonant points. Moreover, if $U$ is a component of $M-K$ then the flow restricted to $U$ is either locally attracted by $K$ (i.e. all points lying in $U$ near $K$ are attracted by $K$) or locally repelled by $K$. Furthermore, if $N$ is an isolating block of $K$ of the form $N=N^+\cup N^-$ then different components of $\partial N$ lie in different components of $M-K$. As a consequence, connected isolated non-saddle sets $K$ of flows in the Euclidean space $\mathbb{R}^n$, $n\geq 2$, have no dissonant points and the flow on every component of $\mathbb{R}^n-K$ is either locally attracted
or locally repelled by $K$. In particular, if $\mathcal{I}(K)=\mathbb{R}^n$ then $K$ is a stable global attractor or a negatively stable global repeller.
\end{theorem}

\begin{proof}
Since $K$ is isolated then $M-K$ consists of a finite number of components. We will study the behaviour of the flow on every component of $M-K$. Suppose $U$ is one of such components. Let $\mathscr{C}$ be the set whose elements are all the components of $M-K$ different from $U$. We define 
\[
\hat{K}=K\bigcup\left(\bigcup_{V\in\mathscr{C}}V\right)
\]

Then $\hat{K}$ is a connected isolated invariant non-saddle set of the flow with $M-\hat{K}=U$. Consider an isolating block $N$ of $\hat{K}$ such that $N=N^{+}\cup N^{-}.$ We will prove that $\hat{K}$ does not disconnect $\mathring{N}$. By considering the initial part of the long exact cohomology sequence of the pair $(M,M-\hat{K})$ we obtain
\[
\ldots\leftarrow H^{1}(M)=\{0\}\leftarrow H^{1}(M,M-\hat{K})\leftarrow \widetilde{H}
^{0}(M-\hat{K})=\{0\}\leftarrow \widetilde{H}^{0}(M)
\]
and, as a consequence, $H^{1}(M,M-\hat{K})=\{0\}.$ By excising from $M$ the compact set $L=M-\mathring{N}$ we get $H^{1}(\mathring{N},\mathring{N}-\hat{K})=H^{1}(M-L,(M-L)-\hat{K})\cong H^{1}(M,M-\hat{K})=\{0\}.$ Hence, from the the
initial part of the long exact cohomology sequence of the pair $(\mathring{N},\mathring{N}-\hat{K})$ we get 
\[
\ldots\leftarrow H^{1}(\mathring{N},\mathring{N}-\hat{K})=\{0\}\leftarrow \widetilde{H}^{0}(%
\mathring{N}-\hat{K})\leftarrow \widetilde{H}^{0}(\mathring{N})=\{0\}
\]
and, as a consequence, $\widetilde{H}^{0}(\mathring{N}-\hat{K})=\{0\}$ which proves that $\hat{K}$ does not {dis\-co\-nnect} $\mathring{N}$. Hence, since $N$ can be chosen to be arbitrarily small, Proposition~\ref{small} ensures that either $\hat{K}$ is an attractor in which case all points of $N$ in $U$ are attracted to $K$ or $\hat{K}$ is a repeller, in which case all points of $N$ in $U$ are repelled by $K$. On the other hand, a similar argument proves that different components of $\partial N$ lie in different components of $M-K.$

Moreover $K$ has no homoclinic orbits and thus no dissonant points. The claim about $\mathbb{R}^{n}$ is readily obtained by
considering its Alexandrov compactification i.e. the $n$-dimensional sphere $S^{n}$ and applying to $S^{n}$ the previous part of the
theorem.
\end{proof}

A nice consequence in the case of unstable attractors in $\mathbb{R}^n$ is the following result from \cite{Moron}.

\begin{corollary}
All connected isolated unstable attractors of flows on $\mathbb{R}^n$ have external explosions.
\end{corollary}

The next theorem gives a local sufficient condition for an isolated {in\-va\-riant} continuum to be non-saddle in terms of its Conley index. 

\begin{theorem}\label{conley}
Let $M$ be a connected and locally compact metric space and suppose that $K$ is an isolated invariant continuum of a flow $\varphi$ on $M$ such that $K$
disconnects a connected neighborhood $W$ of $K$ in $M$ into two components. If $CH_{+}^{1}(K)$ and $CH_{-}^{1}(K)$ are trivial then $K$ is non-saddle and it is neither an attractor nor a repeller. Moreover, if $M$ is a compact manifold without boundary such that $K$ does not disconnect $M$ then $H^{1}(M)\neq \{0\}.$
\end{theorem}

\begin{proof}
Let $W$ be a connected neighborhood of $K$ in $M$ such that $W-K$ consists of two different connected components, $C_{1}$ and $C_{2}$. Consider a
connected isolating block $N$ of $K$ contained in $W$ with entrance and exist sets $N^{i}$ and $N^{o}$ respectively. We remark that if $N^{o}$ is
empty then $N^{i}$ is necessarily non-empty and consisting of at least two components (one at least in each $C_{i})$. Suppose that $N^{o}$ is non-empty.
Since $CH_{+}^{1}(K)=\{0\}$ we have that $\check{H}^{1}(N,N^{o})$ is trivial and from the cohomology exact sequence of the pair $(N,N^{o})$
\[
\ldots\leftarrow \check{H}^{1}(N,N^{o})=\{0\}\leftarrow \widetilde{\check{H}^{0}}(N^{o})\leftarrow \widetilde{\check{H}^{0}}(N)=\{0\}
\]
we get that $\widetilde{\check{H}^{0}}(N^{o})=\{0\}$ and, as a consequence $N^{o}$ is connected. A similar argument using the fact that $CH_{-}^{1}(K)$ is trivial would establish that if $N^{i}$ is non-empty then it is also connected. From this, we get that if both $N^{i}$ and $N^{o}$ are non-empty then the boundary 
$\partial N$ consists of exactly two connected components $N^{o}$ and $N^{i}$ with one of them $N^{o}$ contained in, say, $C_{1}$, and the other, $N^{i}$,
contained in $C_{2}$. The previous argument can be used to exclude the fact that one of the sets $N^{i}$ or $N^{o}$ is empty since, in that case, the
other would be non-empty and consisting of more than one component and we would have a contradiction. Then $K$ can be neither an attractor
nor a repeller and it readily follows that $N$ is an isolating block for $K$ with the structure $N=N^{+}\cup N^{-}$ and, thus, that $K$ is non-saddle. On the other and, if $M$ is a manifold not disconnected by $K$ and $%
H^{1}(M)=\{0\}$ then, by Theorem~\ref{structure}, $K$ attracts or repells all the points of one of its neighborhoods in $M$ and this is a contradiction with the fact just proved that $K$ is neither an attractor nor a repeller.
\end{proof}

The next example is a modification of \cite[Example~2]{Moron} which shows how to produce plenty of examples of flows having an isolated invariant continuum in the conditions of Theorem~\ref{conley}. 

\begin{example}
Let $K$ be a compact and connected manifold (without boundary) endowed with a flow $\varphi_1$ and consider the unit interval $[0,1]$ together with a dynamical
system $\varphi_2$ which has $0$ and $1$ as fixed points and otherwise moves points away from $0$ and towards $1$. The product flow $\varphi(x, s, t) := (\varphi_1 (x, t), \varphi_2 (s,t))$ in the phase space $K\times [0, 1]$ restricts to $\varphi_1$ on $K \times\{0\}$ and $K\times\{1\}$, hence these can be identified to get a flow on the quotient space $K \times S^1$ (with the obvious identifications). Observe that given a closed product neighborhood $W$ of $K$, i.e. a neighborhood mapped onto $K\times [-1,1]$ by a homeomorphism $h:W\to K\times [-1,1]$ such that $h(K)=K\times\{0\}$, $W$ is disconnected by $K$ into two components. Moreover,  $W$ can be chosen  to be an isolating block with entrance and exit set corresponding to $K\times\{-1\}$ and $K\times\{1\}$ respectively. Then, $CH_+^1(K)=CH_{-}^1(K)=\{0\}$ and Theorem~\ref{conley} ensures that $K$ is non-saddle.  A flow $\widehat{\varphi}$ in the conditions of Theorem~\ref{conley} having dissonant points can be obtained modifying $\varphi$. Indeed, let $S$ be a proper closed subset of the exit set of $W$. By using a theorem of Beck \cite{beck} $\varphi$ can be modified to a new flow in such a way that all the orbits of $\varphi$ not meeting a point of $S$ are preserved while the orbits containing a point of $S$ are decomposed in two orbits together with that point of $S$. After this modification $W$ is no longer an isolating block but it is possible to build isolating blocks of $K$ as before in the interior of $W$. It is straightforward to see that $S$ contains externally dissonant points.
\end{example}

We will study now the general structure of a flow on a compact ANR {ha\-ving} an isolated non-saddle set. The next result gives an overall picture of the situation.

\begin{theorem}
Let $K$ be a connected isolated non-saddle set of a flow on a  compact and connected ANR, $M$, (in particular on a compact and connected manifold) and let $L$ be the complement $M-\mathcal{I}(K)$ of its region of influence. Then $L$ is an isolated invariant compactum which is non-empty when $K$ is not an unstable global attractor.  The saddle components of $L$ are exactly those containing externally dissonant points of $K.$ The union of these components is an isolated invariant (saddle) compactum $L_{s}$ and $L-L_{s}$ is an isolated invariant non-saddle compactum that we denote by $L_{n}.$ Moreover, if $x$ is a non-homoclinic point in $\mathcal{I}(K)-K$ and the component of $\mathcal{I}(K)-K$ containing $x$ contains also homoclinic points then either $\omega (x)\subset L_{s}$ or $\omega^* (x)\subset L_{s}.$
\end{theorem}

\begin{proof}
Obviously, $L$ is compact and invariant being the complement in $M$ of the open invariant set $\mathcal{I}(K)$ and it is non-empty when $K$ is not an unstable global attractor. Suppose that $L$ is non-empty since otherwise there is nothing to prove.  Then, if $U$ is a closed neighborhood of $L$ with $U\cap K=\emptyset $ then the trajectory of every point in $U-L$ is contained in $\mathcal{I}(K)$ and hence either its $\omega$-limit or its $\omega^*$-limit is in $K,$ which implies that this trajectory is not entirely contained in $U$. Hence $L$ is isolated.

Suppose that $C$ is a component of $L$. We see that $C$ contains externally dissonant points if and only if $C$ is saddle.  Suppose that $C$ is saddle. Then there exists a neighborhood $U$ of $C$ disjoint from $K$  and a sequence of points  $x_n\in U$, $x_n\to C$ such that $\gamma^+(x_n)\nsubseteq U$ and $\gamma^-(x_n)\nsubseteq U$. Choose an isolating block $N$ of $L$ such that the component $N_C$ of $N$ containing $C$ is contained in $U$. Since for each $n$, $x_n$ belongs to $\mathcal{I}(K)$ we may assume without loss of generality that either $\omega(x_n)\subset K$ for each $n$ or $\omega^*(x_n)\subset K$ for each $n$. We consider the first situation. The trajectory of $x_n$ abandons $N_C$ in positive and negative time since $N_C\subset U$. Let $y_n$ be the sequence corresponding to the exit points of $\gamma^+(x_n)$. We may assume that $y_n\to y\in \partial N_C$ and a simple argument shows that $\omega^*(y)\subset L$. As a consequence $\omega^*(x_n)\subset K$ for almost all $n$ since if not, there is a contradiction with the fact that $\mathcal{A}^*(K)$ is closed in $\mathcal{I}(K)-K$. This proves that $C$ contains externally dissonant points. The converse statement is straightforward and it is left to the reader.

We see that $L_{s}$ is a compactum. Otherwise there are points $x_{n}$ in $L_{s}$ converging to a point $x\in L_{n}$ and hence $x$ belongs to a non-saddle component $C$ of $M-\mathcal{I}(K).$ Consider an arbitrary neighborhood $U$ of $C$ not meeting $K$. Then there is an arbitrarily small neighborhood $U^{\prime }\subset U$ of $C$ such that if a component of $L$ meets $U^{\prime }$ then it is entirely contained in $U^{\prime }.$ Hence $U^{\prime }$ contains the components of some $x_{n}$ and, as a consequence, dissonant and, thus, homoclinic points whose orbit leaves $U$ in the past and in the future. Hence $C$ cannot be non-saddle.

We see now that every component $C$ of $L_{n}$ is isolated. Consider a closed neighborhood $U$ of $L$ such that the component $U_{0}$ containing $C$ does not meet $L_{s}\cup K.$ Then $U_{0}\cap L$ is a compactum and $U_0$ can be chosen in such a way that $U_\cap L$ consists entirely of non-saddle components. Hence $U_{0}\cap L$ is non-saddle.
Moreover $U_{0}\cap L$ is isolated, hence it consists of a finite number of components all isolated non-saddle. On the other hand, we will prove that $L_{n}$ itself consists of a finite number of components. Consider an isolating block $N$ of $K$ in $M$. We know that its boundary, $\partial N,$ consists of a finite number of components. Let $C$ be a component of $L_{n}$ and $S$ the set of points of $\partial N$ which are attracted or repelled by $C.$ We
will see that $S$ is open-closed in $\partial N$ and, hence, $S$ attracts or repells complete components of $\partial N.$ We prove that $S$ is a closed subset of $\partial N$. Suppose $x_{n}\in S$ are attracted by $C$ and $x_{n}\rightarrow x$. All the points $x_{n}$ and $x$ itself belong to $\mathcal{R}^*(K)$. If $x$ is not attracted by $C$ then $\omega (x)$ is contained in a component $C^{\prime }\subset L_{s}$. Consider an isolating block $N$ of $L$ such that $C$ and $C^{\prime }$ lie in different components $N_{C}$ and $N_{C^{\prime }}$ of $N.$ We have points $x_{n}s_{n}\in N_{C^{\prime }}^{o}$ with $s_{n}\rightarrow \infty $. The points $x_{n}s_{n}$ converge to $y\in N_{C^{\prime }}^{o}$ with $ys\in N_{C^{\prime }}$ for every negative $s$. Hence $\omega ^{\ast }(y)\subset N_{C^{\prime }}$ and $y\notin \mathcal{R}(K).$ Hence $\mathcal{R}^{\ast }(K)$ is not closed in $\mathcal{I}(K)-K$. This contradiction proves that $S$ is open-closed in $\partial N$ and hence $C$ attracts or repells complete components. Since every $C$ must attract or repell some component we deduce that there is a finite number of non-saddle components and, thus, $L_n$ is isolated non-saddle.

 If the component of $\mathcal{I}(K)-K$ containing $x$ contains homoclinic points then the component of $\partial N$ containing its orbit contains homoclinic points and by the previous discussion it can be neither attracted nor reppelled by $L_n$. Then, either $\omega(x)\subset L_s$ or $\omega^*(x)\subset L_s$.
\end{proof}

\begin{corollary}
$M-\mathcal{I}(K)$ is non-saddle if and only if the flow does not have dissonant points.
\end{corollary}

We say that an orbit $\gamma\subset\mathcal{I}(K)-K$ is \textit{deviant} if either $\omega (\gamma)$ or $\omega^{\ast }(\gamma )$ contains an externally dissonant point.

All the externally dissonant points are contained in $L_{s}$. In the
important particular case in which all the points of $L_{s}$ are dissonant we have

\begin{corollary}
If $L_{s}$ consists entirely of dissonant points then $\hat{K}=$ $\mathcal{H}(K)\cup \mathcal{I(}L_{s})$ is the smallest simple non-saddle set containing  $K$ i.e. $\hat{K}$ is obtained from $K$ by adding all the homoclinic orbits plus all the deviant orbits.
\end{corollary}

\section{Non-saddle sets in 2-dimensional flows}

In this section we will study some dynamical properties of non-saddle sets of flows defined on $2$-dimensional manifolds. In particular, the following result shows that in the case of closed orientable surfaces, the existence of dissonant points can be detected by looking at some relations between the topologies of $K$ and its region of influence.  Since in this section we are dealing with either $2$-dimensional manifolds or the Euclidean plane $\mathbb{R}^2$, Gutierrez' Theorem allows us to assume that all the flows considered are $C^1$ and $C^\infty$ respectively.

\begin{theorem}
An isolated non-saddle continuum $K$ of a flow on a closed {orien\-ta\-ble} surface $M$ does not have dissonant points if and only if the homology and cohomology of $\mathcal{I}(K)$ are finitely generated and $\chi(K)=\chi(\mathcal{I}(K))$.
\end{theorem}

\begin{proof}
Suppose that $K$ has no dissonant points. It follows from Alexander duality theorem that 
\[
H_k(\mathcal{I}(K),\mathcal{I}(K)-K)\cong\check{H}^{2-k}(K)
\]
and, since $K$ has finitely generated \v{C}ech cohomology, $\chi(\mathcal{I}(K),\mathcal{I}(K)-K)$ is defined and agrees with $\chi (K)$. Moreover, if $K$ does not have dissonant points then $\mathcal{I}(K)$ and $\mathcal{H}(K)$ are of the same shape, which implies, since $\mathcal{H}(K)$
is non-saddle, that $\chi (\mathcal{I}(K))$ is well defined. As a consequence, $\chi(\mathcal{I}(K)-K)$ is also defined and 
\begin{align*}
\chi (\mathcal{I}(K))=\chi (\mathcal{I}(K),\mathcal{I}(K)-K)+\chi (
\mathcal{I}(K)-K)= \\
\chi (K)+\chi (\mathcal{I}(K)-K).
\end{align*}%
Consider now an isolating block $N$ of $K$ in $M$ of the form $N=N^{+}\cup N^{-}$. We may suppose by \cite{Conley-Easton} that $N$ is a compact surface whose boundary $\partial N$ consists of a finite union of circles. Then there exists a strong deformation retraction of $\mathcal{I}(K)-K$ onto some of those
circles (especifically onto the union of $N^{o}$ with $N^{i}-\mathcal{H}(K)$) and it follows that $\chi (\mathcal{I}(K)-K)=0$. As a consequence, $\chi(\mathcal{I}(K))=\chi(K).$

Conversely, suppose that the Euler characteristic of $\mathcal{I}(K)$ is defined and $\chi(K)=\chi(\mathcal{I}(K))$. Then $\chi(\mathcal{I}(K)-K)=0 $. Moreover, $\mathcal{I}(K)-K$ is a disjoint union of connected orientable surfaces $S_{1},\ldots,S_{n}$, all of them proper open subsets of the closed
surface $M.$ If $S_{i}$ does not contain dissonant points then there exists a strong deformation retraction of $S_{i}$ onto a component of $\partial
N\cap S_{i}$ (a circle) and, thus, $\chi (S_{i})=0$. Now suppose, to get a contradiction, that $S_{i}$ contains dissonant points. Since it is a proper
subset of $M$ \ then $\chi (S_{i})\leq 1$. We will see that the possibilities $\chi (S_{i})=0$ or $1$ are excluded. If $\chi (S_{i})=0$ then
by \cite{Richards} $S_{i}$ is homeomorphic to a punctured open disk $D-\{p\}$. If $C$ is a component of $\partial N\cap S_{i}$ then $C$ is a circle not contractible
in $S_{i}$ (otherwise the disc bounded by $C$ would be positively or negatively invariant by the flow and would contain an invariant set not
influenced by $K $). The intersection $\partial N\cap S_{i}$ consists of more than one component since, otherwise, there would be no homoclinic
orbits in $S_{i}$ and, thus, no dissonant points. The external and the internal components of $\partial N\cap S_{i}$ limit a region $R$ in $D-\{p\}$
homeomorphic to a planar ring. All the trajectories of points of $R$ abandon $R$ in the past and in the future (otherwise there would exist an invariant
set in $R$ not influenced by $K$). This implies that all the trajectories in $S_{i}$ are homoclinic and, hence, there are no dissonant points. A
similar but easier argument excludes the possibility that $\chi (S_{i})=1.$ Hence $\chi (S_{i})<0$ for every surface containing dissonant points. Since $%
\chi (\mathcal{I}(K)-K)=\sum_{i=1}^{n}\chi (S_{i})=0$ and all the surfaces $S_{i}$ without dissonant points are of zero Euler characteristic, we
conclude that there are no dissonant points in $\mathcal{I}(K)-K.$
\end{proof}

The next result establishes a relation between the homoclinic orbits of a plane continuum and the existence of fixed points in its complement.

\begin{lemma}\label{intersection}
Suppose that $K$ is an isolated invariant continuum of a plane flow and $\mathcal{H}(K)-K\neq \emptyset $. Then there exists a fixed point in $\mathbb{R}^{2}$ $-K.$ More {ge\-ne\-ra\-lly}, if $K$ is an isolated invariant continuum of a plane flow and a component $U$ of $\mathbb{R}^{2}$ $-K$
contains a trajectory $\gamma $ such that $\omega (\gamma )\cap K\neq\emptyset $ and $\omega ^*(\gamma )\cap K\neq \emptyset $ then there
exists a fixed point in $U$.
\end{lemma}

\begin{proof}
 Suppose, to get a contradiction, that there exists a trajectory $\gamma $ in a component $%
U$ of $\mathbb{R}^{2}$ $-K$ such that $U$ does not contain fixed points and $\omega (\gamma )\cap K\neq \emptyset $ and $\omega ^*(\gamma )\cap
K\neq \emptyset .$ Take an isolating block $N$ of $K$. By \cite{Conley-Easton} $N$ can be chosen to be a topological closed disk with $i$ holes, one for every bounded
component of $\mathbb{R}^{2}-K$. We suppose that $U$ is the unbounded component (the argument being only slightly diferent in the other case) and
consider the only circle $C\subset \partial N$ contained in $U$. Then there exists a point $x\in C\cap \gamma $ leaving $N$ and returning to $N$ after a
time $t\neq 0$, i.e. such that $xt\in C$ and $x(0,t)\cap N=\emptyset .$ The possibility that the time $t$ be positive or negative is irrelevant in this
construction. Consider the arc $A$ in $C$ with extremes $x$ and $xt$ such that the topological circle $x[0,t]\cup A$ does not contain $K$ in its
interior. This arc can be mapped to the unit interval $I=[0,1]$ of the real line by a homeomorphism $h:A\rightarrow I$. If we take the point $x_{1}\in 
\mathring{A}$ corresponding to the center of $I$ then $x_{1}$ must leave $N$ (in the past or in the future) and return again since, otherwise, the
Theorem of Poincar\'{e}-Bendixson would imply the existence of a fixed point in the disk limited by $x[0,t]\cup A$. Hence we can repeat the operation
with $x_{1}[0,t_{1}]\cup A_{1}$, where $A_{1}$ is an arc in $A$ with extremes $x_{1}$ and $x_{1}t_{1}$ and the topological circle $%
x_{1}[0,t_{1}]\cup A_{1}$ does not contain $K$ in its interior. Now take $x_{2}\in $ $\mathring{A}_{1}$ corresponding to the middle point of $h(A_{1})$
and repeat the construction. In this way we obtain a sequence $A\supset A_{1}\supset A_{2}\supset\ldots$ of arcs whose intersection $\bigcap_{i=1}^\infty A_{i}$ consists of one point $p\in \partial N$. The orbit of $p$ defines an internal tangency to $\partial N,$ which is in contradiction with the
properties of isolating blocks. This contradiction proves that if $\omega(\gamma )\cap K\neq \emptyset $ and $\omega ^*(\gamma )\cap K\neq \emptyset $ then there exists a fixed point in $U$.
\end{proof}

As a consequence of the last proposition a lower bound for the number of fixed points in the complement of an isolated invariant plane continuum is obtained.

\begin{corollary}
Let $K$ be an isolated invariant continuum of a plane flow. {Suppo\-se} that $\mathbb{R}^{2}$ $-K$ has $i$ connected components. Then, there are at least $%
i-1$ fixed points in $\mathbb{R}^{2}-K.$
\end{corollary}

\begin{proof}
 We see that, in fact, there is at least one fixed point in every bounded component $U$ of $\mathbb{R}^{2}$ $-K$. Otherwise, if $\gamma $ is a
trajectory in the bounded component $U,$ where $U$ does not contain fixed points, then by Lemma~\ref{intersection} either $\omega(\gamma)\cap K=\emptyset$ or $\omega^*(\gamma)\cap K=\emptyset$. Hence, Poincar\'e-Bendixson Theorem ensures the existence of a periodic orbit contained in $U$, and thus the existence of a fixed point in its interior. This leads to a contradiction with the assumption, since the interior of any periodic orbit contained in $U$ is also contained in $U$. The number of bounded components of $\mathbb{R}^2-K$ is exactly $i-1$, so this contradiction establishes the corollary.
\end{proof}

In \cite{Barge-Sanjurjo} it was proved that the non-existence of fixed points in an isolated invariant continuum $K$ of a planar flow is a sufficient condition for $K$ to be non-saddle. We see in the next result that an {a\-ssump\-tion} about the region of influence of an isolated invariant plane continuum is also  sufficient to guarantee its non-saddleness.

\begin{theorem}\label{plane}
Let $K$ be an isolated invariant continuum of a flow on $\mathbb{R}^{2}$ and suppose that $\mathcal{I}(K)=\mathbb{R}^{2}.$ Then $K$ is non-saddle and, as
a consequence, a stable global attractor or a negatively stable global repeller.
\end{theorem}

\begin{proof}
 Suppose, to get a contradiction, that $K$ is saddle. Then there exist an isolating block $N$
which, by the previous corollary can be taken to be a topological closed disk, and a sequence of points $x_{n}\rightarrow K$ whose trajectories leave $N$ in the future
and in the past. Denote by $y_{n}$ and $z_{n}$ the corrresponding exit points (in the future and in the past respectively). We may assume that $y_{n}\rightarrow y_{0}\in \partial N$ and $z_{n}\rightarrow z_{0}\in\partial N$. A simple argument proves that $\emptyset\neq\omega ^*(y_{0})\subset K$
and $\emptyset\neq\omega (z_{0})\subset K$. Since $\mathcal{I}(K)=\mathbb{R}^{2}$ then we may suppose that either $\emptyset\neq\omega (y_{n})\subset K$ for almost all $n$ or $º\emptyset\neq\omega^*(y_{n})\subset K$ for almost all $n$; we consider the first situation and suppose, for the sake of simplicity in notation, that $\emptyset\neq\omega (y_{n})\subset K$ for $n\geq 1$. We may also assume that all points $y_{n}$ and $y_{0}$ are contained in an arc $A\subset N^{o}$ with no tangency points. Consider now an arbitrary isolating block $N_{1}\subset \mathring{N}$ \ which is also a disk. We can define a topological circle $C$, not having $K$ in its interior, consisting of the union of the following sets: a) the trajectory $\gamma ^{-}(y_{0})$ until it reaches $\partial N_{1}$ in a point $a$, b) an arc $J$ $\subset A\subset N^{o}$,
linking $y_{0}$ to $y_{1}$, c) the trajectory $\gamma ^{+}(y_{1})$ until it reaches $\partial N_{1}$ in a point $b$ and d) an arc in $\partial N_{1}$
linking $a$ to $b$. Denote by $D$ the disk bounded by $C$.
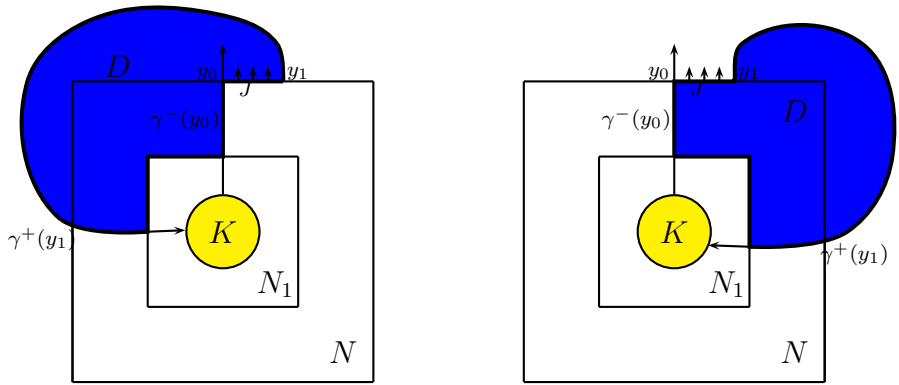
\begin{figure}[h]
\begin{pspicture}(-3,-2)(11,3)

%%%Cuadrado 1%%%
\psline(-2,-2)(2,-2)
\psline(-2,-2)(-2,2)
\psline(2,-2)(2,2)
\psline(-2,2)(2,2)
\psline(-1,-1)(-1,1)
\psline(-1,-1)(1,-1)
\psline(-1,1)(1,1)
\psline(1,-1)(1,1)
\pscustom[fillstyle=solid,fillcolor=yellow]{
\pscircle(0,0){0.5}}

%%%Cuadrado 2%%
\psline(4,-2)(8,-2)
\psline(4,-2)(4,2)
\psline(8,-2)(8,2)
\psline(4,2)(8,2)
\psline(5,-1)(5,1)
\psline(5,-1)(7,-1)
\psline(5,1)(7,1)
\psline(7,-1)(7,1)
\pscustom[fillstyle=solid,fillcolor=yellow]{
\pscircle(6,0){0.5}}

%%%%Lineas cuadrado 2%%%

\psline{->}(6,0.5)(6,2.5)
\psline{->}(6.2,2)(6.2,2.2)
\psline{->}(6.4,2)(6.4,2.2)
\psline{->}(6.6,2)(6.6,2.2)

\pscustom[linewidth=1.5pt,fillstyle=solid,fillcolor=blue]{
\psline(6,1)(6,2)
\psline(6,2)(6.8,2)
\pscurve(6.8,2)(6.9,2.5)(8.5,2.5)(8.9,1)(8.3,0)(7,-0.2)  
\psline(7,-0.2)(7,1)   
\psline(7,1)(6,1)
}
\psline{->}(7,-0.2)(6.45,-0.18)

\psline(8,-2)(8,2)
\psline(6,2)(8,2)

%%%Lineas cuadrado 1$$

\pscustom[linewidth=1.5pt,fillstyle=solid,fillcolor=blue]{
\pscurve(0.8,2)(0.72,2.5)(-0.1,2.9)(-1.8,2.8)(-2.6,2)(-2.2,.2)(-1,0)
\psline(-1,0)(-1,1)
\psline(-1,1)(0,1)
\psline(0,1)(0,2)
\psline(0,2)(0.8,2)
}

\psline{->}(-1,0)(-0.5,0.02)
\psline(-2,-2)(-2,2)
\psline(-2,2)(0,2)
\psline{->}(0,0.5)(0,2.5)
\psline{->}(0.2,2)(0.2,2.2)
\psline{->}(0.4,2)(0.4,2.2)
\psline{->}(0.6,2)(0.6,2.2)
%\psline(7.8,1.8)(7.8,2)

%\psline[linewidth=1.5pt](0,1)(0,2)
%\psline[linewidth=1.5pt](-0.02,2)(0.82,2)
%\pscurve[linewidth=1.5pt](0.8,2)(0.9,2.5)(2,2.7)(2.5,2.5)(2.9,2)(2.8,1)(2,0)(1,-0.2)     
%\psline[linewidth=1.5pt](-0.02,1)(1.025,1)
%\psline[linewidth=1.5pt](1,1)(1,-0.22)

%%%%Texto cuadrado 1%%%%
\rput(-0.2,2.1){\tiny $y_0$}
\rput(0.3,1.9){\tiny $J$}
\rput(1,2.1){\tiny $y_1$}
\rput(-0.5,1.5){\tiny $\gamma^-(y_0)$}
\rput(-2.4,-0.1){\tiny $\gamma^+(y_1)$}
\rput(0,0){$K$}
\rput(1.6,-1.6){$N$}
\rput(0.7,-0.7){$N_1$}
\rput(-1.4,2.2){$D$}

%%%%Texto cuadrado 2%%%%
\rput(5.8,2.1){\tiny $y_0$}
\rput(6.3,1.9){\tiny $J$}
\rput(7,2.1){\tiny $y_1$}
\rput(5.5,1.5){\tiny $\gamma^-(y_0)$}
\rput(8.4,-0.3){\tiny $\gamma^+(y_1)$}
\rput(6,0){$K$}
\rput(7.6,-1.6){$N$}
\rput(6.7,-0.7){$N_1$}
\rput(7.6,1.6){$D$}
\end{pspicture}
\caption{The region $D$}
\end{figure}

 Then $\mathring{J}$ is either an exit set or an entrance set for $D.$ If $\mathring{J}$ is an exit set then either $\gamma^{-}(y_{1})$ is contained in $D$ or there exists a point of $\gamma^{-}(y_{1})$ contained in $\partial N_{1}$. It is not difficult to see, using the Poincar\'{e} Bendixson Theorem that the first case is impossible. If $\mathring{J}$ is an entrance set for $D$ then either $\gamma ^{+}(y_{0})$ is contained in $D$ or there exists a point of $\gamma^{+}(y_{0})$ contained in $\partial N_{1}$. The first case is impossible as before. We conclude that either there exists a point of $\gamma ^{-}(y_{1})$ contained in $\partial N_{1}$ or there exists a point of $\gamma ^{+}(y_{0})$ contained in $\partial N_{1}$. If we repeat this construction for a sequence of isolating neighborhoods $N_{i}$ shrinking to $K$ then we get that there exists a point of $K$ in the $\omega$-limit of $y_{0}$ or a point of $K$ in the $\omega ^*$-limit of $y_{1}.$ Then, as a consequence of Proposition~\ref{intersection}, there would exist a fixed point in $\mathbb{R}^{2}-K$, which cannot be in the region of influence of $K$. It follows from this contradiction that $K$ is non-saddle and, as a consequence, a global atractor or a global repeller.
\end{proof}

A direct consequence of Theorem~\ref{plane} is the following result from \cite{Moron}.

\begin{corollary}[Mor\'on, S\'anchez-Gabites and Sanjurjo  \cite{Moron}]
Let $K$ be an isolated invariant continuum of a flow on $\mathbb{R}^2$ and suppose that $\mathcal{A}(K)=\mathbb{R}^2$. Then $K$ is stable and, thus, a global attractor.
\end{corollary}

\section{Robustness of non-saddle sets}

It was shown in \cite{GS Top} that the property of being non-saddle is not robust, i.e. it is not preserved by continuation of
isolated invariant sets. However, it turns out that there exist some relations between the preservation of certain topological properties by continuation and the preservation of the dynamical property of non-saddleness. As a matter of fact, in some situations both properties are equivalent, as we see in the
following result. We introduce first the necessary definitions.

\begin{definition} 
Suppose $\varphi_\lambda:M\times\mathbb{R}\to M$ is a parametrized family of flows (parametrized by $\lambda\in I$, the unit interval) in
a locally compact ANR, $M$, and suppose that $K_0$ is an isolated non-saddle set for $\varphi_0$. We say that $K_0$ is \emph{{dy\-na\-mi\-ca\-lly} robust} if
for every isolating neighborhood $N$ of $K_0$ there exists $\delta > 0$ such that, for every $\lambda\in [0, \delta)$, the isolated invariant subset $K_\lambda$ of
$N$ (with respect to the flow $\varphi_\lambda$) which has $N$ as an isolating neighborhood is a (non-empty) non-saddle set.
\end{definition}

By \cite[Lemma~6.2]{Salamon}, we have that $K_0$ is dinamically robust if and only if there exist an isolating neighborhood $N$ of $K_0$ and a $\delta > 0$ such that, for every $\lambda\in[0, \delta)$, the isolated invariant subset $K_\lambda$ of $N$ (with respect to the flow $\varphi_\lambda$) which has $N$ as an isolating neighborhood is a (non-empty) non-saddle set.

\begin{definition}
Suppose $\varphi_\lambda:M\times\mathbb{R}\to M$ is a parametrized family of flows (parametrized by $\lambda\in I$, the unit interval) in
a locally compact ANR, $M$, and suppose that $K_0$ is an isolated invariant set for $\varphi_0$. We say that $K_0$ is \emph{topologically} robust
if for every isolating neighborhood $N$ of $K_0$ there exists $\delta>0$ such that, for every $\lambda\in [0, \delta)$, the isolated invariant subset $K_\lambda$
of $N$ (with respect to the flow $\varphi_\lambda$) which has $N$ as an isolating neighborhood has the same shape as $K_0$.  
\end{definition}

By \cite[Lemma~6.2]{Salamon}, we have that $K_0$ is topologically robust if and only if there exist an isolating
neighborhood $N$ of $K_0$ and a $\delta > 0$ such that, for every $\lambda\in[0, \delta)$, the isolated invariant subset $K_\lambda$ of $N$ (with respect to the
flow $\varphi_\lambda$) which has $N$ as an isolating neighborhood has the same shape as $K_0$.

Note that when a non-saddle set is dynamically robust, this fact implies the existence of a (local) continuation made
of non-saddle sets. On the other hand, if an isolated invariant set is topologically robust, then it has a (local) continuation
whose members have the same shape.

\begin{theorem}\label{robustness}
Let $\varphi _{\lambda },$ with $\lambda \in \lbrack 0,1]$, be a differentiable parametrized family of flows on a closed and connected orientable differentiable manifold $M$ with $H^{1}(M)=\{0\}$ and $K_{0}$ be a connected isolated non-saddle set. Then $K_0$ is dinamically robust if and only if it is topologically robust.
\end{theorem}

\begin{proof}
It has been proved in \cite{GS Top} that dynamical robustness implies topological robustness. We will prove here the converse implication. Suppose that $K_0$ is topologically robust, i.e that $Sh(K_{\lambda })=Sh(K_{0})$ for $\lambda < \delta.$ We see that since $H^{1}(M)=\{0\}$ then all $M-K_{\lambda }$ have the same number of components. Consider the terminal part of the cohomology long exact sequence
\[
\ldots\leftarrow H^{1}(M)=\{0\}\leftarrow H^{1}(M,M-K_{\lambda })\leftarrow 
\widetilde{H}^{0}(M-K_{\lambda })\leftarrow \widetilde{H}^{0}(M)=\{0\}.
\]

Then $H^{1}(M,M-K_{\lambda })\cong \widetilde{H}^{0}(M-K_{\lambda })$ and $H^{1}(M,M-K_{\lambda })$ is free and finitely generated and, by the Universal Coefficient Theorem, $\rank H^{1}(M,M-K_{\lambda })=\rank H_{1}(M,M-K_{\lambda })$. Moreover, by Alexander duality, $H_{1}(M,M-K_{\lambda})\cong
\check{H}^{n-1}(K_{\lambda })$ and, since \v{C}ech cohomology is a shape invariant, we have that all $\widetilde{H}^{0}(M-K_{\lambda })$ are the
same for $\lambda < \delta$ and, hence, all $M-K_{\lambda }$ have the same number of components.

Consider a connected isolating block $N$ of $K_{0}$ such that $N$ is a differentiable manifold. It was proved in \cite{GS Top} that $N$ is also an isolating block for all $K_{\lambda }$ with $\lambda $ sufficiently close to $0$ (we may assume that this happens for all $\lambda < \delta$). It was also proved that the entrance and exit sets for $\varphi _{\lambda }$ are strict (i.e. without tangencies) and that they agree with those for $\varphi_{0}.$ Moreover, in Theorem~\ref{structure} it was proved that different components of $\partial N$ lie in different components of $M-K_{0}$. Suppose now, to get a contradiction, that $K_{\lambda }$ is saddle for\ some $\lambda < \delta$. Then there are sequences $x_{n}\rightarrow x\in K_{\lambda },$ $s_n<0<t_{n}$ such that $x_{n}t_{n}\in N^{0}$ and $x_{n}s_{n}\in N^{i}$. This implies that the component of $M-K_{\lambda }$ containing $x_{n}$ contains points of $N^{0}$ and $N^{i}$ and, hence, that there are two components, or more, of $\partial N$ in the component of $M-K_{\lambda }$ which contains $x_n.$ We deduce from this that the number of components of $M-K_{\lambda }$ is less than the number of components of $M-K_{0}$. We get from this contradiction that all $K_{\lambda }$ are non-saddle for $\lambda< \delta$.
\end{proof}

\begin{remark}
It is straightforward to see that Theorem~\ref{robustness} also holds in the case of differentiable parametrized families of flows on $\mathbb{R}^n$, $n\geq 2$. 
\end{remark}

\end{document}